\documentclass[10pt]{amsart}
\usepackage{etex}

\usepackage{amsmath}
\usepackage{amssymb}
\usepackage{amsthm}
\theoremstyle{definition}
\usepackage{amscd}
\usepackage[all,cmtip]{xy}
\usepackage{enumerate}

\usepackage[utf8]{inputenc} 
\usepackage{tikz}
\usetikzlibrary{arrows ,scopes}

\newtheorem{Theorem}{Theorem}

\newtheorem{theorem}{Theorem}[section]
\newtheorem{lemma}[theorem]{Lemma}
\newtheorem{corollary}[theorem]{Corollary}
\newtheorem{proposition}[theorem]{Proposition}
\newtheorem{remark}[theorem]{Remark}
\newtheorem{definition}[theorem]{Definition}
\newtheorem{notation}[theorem]{Notation}
\newtheorem{example}[theorem]{Example}
\newtheorem{question}[theorem]{Question}

\newcommand{\domain}{\M}
\newcommand{\range}{\B}
\newcommand{\AAA}{\mathbb{A}}

\newcommand{\I}{\mathcal{I}}     
  
\newcommand{\E}{\mathbb{E}}    
\newcommand{\EE}{\mathbb{E}}   
\newcommand{\A}{\mathcal{A}}   
\newcommand{\B}{\mathcal{B}}    
\newcommand{\M}{\mathcal{M}}  
\newcommand{\N}{\mathbb{N}}  
\newcommand{\C}{\mathbb{C}}     
\newcommand{\F}{\mathcal{F}}    
 \newcommand{\Q}{\mathcal{Q}}  

\newcommand{\s}{\mathcal{S}}  
\newcommand{\pp}{\mathcal{P}}   

\newcommand{\ii}{{\bf \tilde{i}}}

\newcommand{\dd}{{\bf \tilde{d}}}
\newcommand{\Ak}{\A^{\otimes_\B k}}

\newcommand{\qo}{\mathcal{O}^+}
\newcommand{\qs}{\mathcal{S}^+}
\newcommand{\qh}{\mathcal{H}^+}
\newcommand{\rqh}{\mathcal{H'}^{+}}
\newcommand{\qb}{\mathcal{B}^+}
\newcommand{\qu}{\mathcal{U}^+}

\newcommand{\co}{\mathcal{O}}
\newcommand{\cs}{\mathcal{S}}
\newcommand{\ch}{\mathcal{H}}
\newcommand{\cb}{\mathcal{B}}
\newcommand{\cu}{\mathcal{U}}

\begin{document}
\title[De Finetti Theorems]{ Full classification of de Finetti type theorems for *-random variables in classical and free probability}
\author{Weihua Liu}
\maketitle

\begin{abstract}
Classical distributional symmetries can be described as invariance under the actions of semigroups (or groups) of matrix structures, and subsequently under the coactions of continuous functions on the matrix semigroups (or groups) generated by entry functions. By considering noncommutative entry functions on matrix structures, Woronowicz introduced corepresentations of compact quantum groups, namely Woronowicz's $C^*$-algebras (also known as compact matrix pseudogroups).  We demonstrate that every nontrivial finite sequence of random variables admits a maximal distributional symmetry determined by a Woronowicz $C^*$-algebra. This establishes a probabilistic framework for classifying compact quantum groups.  Furthermore, we classify all de Finetti-type theorems for *-random variables that are invariant under distributional symmetries arising from compact matrix quantum groups in both classical and free probability settings. Our results show that only finitely many types of de Finetti theorems exist in these contexts, and the associated categories of (quantum) groups are the easy (quantum) groups introduced by Banica and Speicher.

\end{abstract}

\section{Introduction}

The classical  permutation groups $S_n$ of $n$ elements and orthogonal groups $O_n$ on $n$-dimensional spaces have natural actions on  $n$ random variables $x_1,\cdots,x_n$. 
For $\sigma \in S_n$, we can send $x_i$ to $x_{\sigma(i)}$ and for $\sigma=(q_{i,j})_{i,j=1,\cdots,n}\in O_n$, we can send $x_i$ to $\sum\limits_{k=1}^n q_{i,k}x_k$.
Random variables are said to be exchangeable (or orthogonal) if their joint distribution is invariant under the action of the corresponding group. 
For exchangeability, we have
$$(x_{1},\cdots,x_{n})\stackrel{d}{=}(x_{\sigma(1)},\cdots,x_{\sigma(n)})$$
for all $\sigma\in S_n$,
and for orthogonality, we have
$$(x_{1},\cdots,x_{n})\stackrel{d}{=}(\sum\limits_{k=1}^n q_{1,k}x_k,\cdots,\sum\limits_{k=1}^n q_{n,k}x_k)$$
for all $\sigma=(q_{i,j})_{i,j=1,\cdots,n}\in O_n.$
These two symmetries on finitely many random variables can be generalized to  infinite sequences of random variables by assuming that all finite subsequences possess the given property.
Note that exchangeability can also be expressed in terms of matrix operations as orthogonality by setting $\sigma=(\delta_{\sigma(i),j})_{i,j=1,\cdots,n}$. 
Based on this observation, we may introduce distributional symmetries for random variables beyond exchangeability and orthogonality by using matrix operations on elements from $n\times n$ matrices.
Given two $n\times n$ matrices  $g=(g_{i,j})$ and $h=(h_{i,j})$ over real numbers, if 
$$(x_{1},\cdots,x_{n})\stackrel{d}{=}(\sum\limits_{k=1}^n g_{1,k}x_k,\cdots,\sum\limits_{k=1}^n g_{n,k}x_k)$$
and 
$$(x_{1},\cdots,x_{n})\stackrel{d}{=}(\sum\limits_{k=1}^n h_{1,k}x_k,\cdots,\sum\limits_{k=1}^n h_{n,k}x_k),$$
then
$$
\begin{array}{rcl}
(x_{1},\cdots,x_{n})&\stackrel{d}{=}&(\sum\limits_{k=1}^n g_{1,k}x_k,\cdots,\sum\limits_{k=1}^n g_{n,k}x_k)\\
&\stackrel{d}{=}&(\sum\limits_{l=1}^n h_{1,l}\sum\limits_{k=1}^n g_{l,k}x_k,\cdots,\sum\limits_{l=1}^n h_{n,l}\sum\limits_{k=1}^n g_{l,k}x_k)\\
&=&\left(\sum\limits_{k=1}^n (hg)_{1,k}x_k,\cdots,\sum\limits_{k=1}^n (hg)_{n,k}x_k\right),
\end{array}$$
where $(hg)_{i,j}$ is the $(i,j)$-th entry of $hg$. Thus, the joint distribution of $(x_{1},\cdots,x_{n})$ is invariant under the matrix operation of the semigroup generated by $g$ and $h$. 
Two other important distributional symmetries, namely stationary and spreadability, can also be defined through matrix operations on semigroups of infinite-dimensional matrices generated by $(\delta_{i,i+1})_{i,j\in N}$ and $\{(\delta_{i,f(i)})_{i,j\in N} | f:\N\rightarrow \N, f\text{ is strictly increasing }\}$, respectively.  
The study of exchangeability dates back to 1930s, when de Finetti showed that infinite exchangeable sequences of random variables taking values in $\{0,1\}$ are conditionally independent and identically distributed. 
This result was later extended to random variables taking values in compact Hausdorff spaces by Hewitt and Savage \cite{HS}. 
On the other hand, Freedman studied sequences of orthogonal random variables and showed that infinite sequences with this property are conditionally independent Gaussian Families \cite{Freedman}. 
For spreadability, Ryll-Nardzewski proved an extended de Finitte type theorem,  infinite spreadable sequences of random variables  are conditionally independent and identically distributed.  

\begin{question} \label{quesition1}Do we have de Finetti-type theorems for distributions other than Gaussian? Additionally, do we have de Finetti-type theorem for a given semigroup of (infinite-dimensional) matrices?
\end{question}

Noncommutative probability as non-classical probability theory was introduced by von Neumann in the early thirties.
Roughly speaking, in a very general setting of noncommutative probability, the set of random variables is replaced by a (noncommutative) algebra, referred to as a probability space, and the expectation is replaced by a linear functional.
Free independence, introduced by Voiculescu to tackle the isomorphism problem of free group von Neumann algebras, serves as a quantum analogue to the independence relations in noncommutative probability. 
Notably, the only two unital universal independent relations for random variables, based on the natural requirements of associativity and a universal calculation rule for mixed moments, are classical independence and free independence \cite{Sp}.
Here, the universal independence relation means that the expectation functional on a probability space is completely determined by the expectation functionals on the sub-probability spaces that satisfy the given relation.
The independence relation in noncommutative probability is more complicated than in classical probability.
 Indeed, if we do not require the universal relation to be unital, one additional type of independence emerges, known as Boolean independence \cite{SW}. 
If we do not require associativity, we will encounter infinitely many independence relations that satisfy a common weak condition known as conditional independence \cite{JL}. 
In addition, we are able to construct infinite exchangeable sequences of noncommutative random variables to satisfy infinitely many symmetric universal independence relation \cite{JL}.
Therefore, in noncommutative probability, classical exchangeability cannot determine a universal independence relation conditionally as they do in the classical sense.  
Even if we strengthen exchangeability to orthogonality, we still do not obtain a de Finetti-type characterization for any universal independence relation.
For example,  an infinite sequence of q-Gaussian random variables is orthogonal but does not satisfy any common universal independence relation \cite{BS}.
In noncommutative probability, a de Finetti-type theorem for exchangeability is as follows: if the probability space generated by exchangeable sequences of random variables behaves well, such as in a W*-probability space with a faithful normal state, then the random variables are conditionally independent \cite{Ko}.
To characterize a given universal independence relation, such as freeness, it requires more distributional symmetries than classical exchangeability. 

In 1987, Woronowicz introduced a notion of compact matrix quantum groups, which provide symmetries for noncommutative structures \cite{Wo2}.  
Given a Lie group $G\subset M_n(\C)$, by Weierstrass theorem, the algebra of continuous function $C(G)$ on $G$ is generated by $\{u_{i,j}|i,j,=1,\cdots, n\}$, where $u_{i,j}(g)=g_{i,j}$ for $g=(g_{i,j})\in M_n(\C)$. 
The map $\Delta(u_{i,j})=\sum\limits_{k=1}^n u_{i,k}\otimes u_{k,j}$ defines a homomorphism from $C(G)$ to $C(G)\otimes C(G)$, which is called the comultiplication.  
By considering noncommutative $C^*$-algebras $\A$ generated by elements $\{u_{i,j}|i,j,=1,\cdots, n\}$ with well-defined comultiplication $\Delta(u_{i,j})=\sum\limits_{k=1}^n u_{i,k}\otimes u_{k,j}$ from $\A$ to $\A\otimes_{\min} \A$ and some other conditions, one gets the notion of a compact matrix quantum group. 
In this framework, Wang introduced the quantum analogues $\A_s(n)$ and $\A_o(n)$ of permutation groups and orthogonal groups, respectively \cite{Wan2}. The distributional symmetries associated with classical Lie groups $G \subset M_n(\C)$ can be extended via the entry functions ${u_{i,j} \mid i,j=1,\ldots,n}$. 
Consequently, the distributional symmetries generated by $\A_s(n)$ and $\A_o(n)$ are applied in noncommutative probability and are referred to as quantum exchangeability and quantum orthogonality.

K\"ostler and Speicher found quantum exchangeability can be used to characterize freeness. 
In fact, they  proved that an infinite sequence of noncommutative selfadjoint random variables  in $W^*$-probability space with a faithful normal state is quantum exchangeable if and only if the random variables are identically distributed and free with respect to the conditional expectation onto their tail algebra \cite{KS}.
Later,  the free analogue of de Finetti theorem for non selfadjoint random variables is proved by Curran \cite{Cu2} and 
the free analogue of Freedman's result associated with Wang's quantum orthogonal groups is also proved by Curran\cite{Cu}.
Here comes a natural question: 
\begin{question}\label{question2} Do we have de Finetti-type theorems for distributions other than Free Gaussian?  What are they?
\end{question}

To answer Question (\ref{quesition1}) and Question (\ref{question2}), we certainly need to consider matrix groups and matrix quantum groups other than (quantum)permutation groups and (quantum)orthogonal groups. 
In \cite{BaS}, by considering tensor categories spanned by certain partitions coming from the tensor category of $S_n$, Banica and Speicher introduced the notion of easy groups and easy quantum groups.  
Together with a work of Weber \cite{Weber1}, We know that there are six easy groups lying between (and including) $S_n$ and $O_n$ and seven easy quantum groups lying between (and including) $A_s(n)$ and $A_o(n)$.  
In \cite{BCS},  Banica, Curran, and Speicher found two de Finetti type theorems besides the original one and Freedman's in classical probability, as well as two free de Finetti type theorems beyond the quantum exchangeable case and the quantum orthogonal case in free probability. 
Later, the author showed that there are exactly four de Finetti type theorems in classical probability and four de Finetti type theorems in free probability in selfadjoint case \cite{Liu3}. 
It shows that  different matrix groups or matrix quantum groups  may yield the same de Finetti theorems. 
In addition, for each de Finetti type theorem in \cite{Liu3}, there exists a maximal distributional symmetry  in the sense that if the random variables satisfy more distributional symmetries beyond the given one, then the de Finetti theorem fails. 
Furthermore, the maximal distributional symmetries are represented by easy groups or easy quantum groups. 
Thus, the classification of maximal distributional symmetry (quantum) groups for infinite sequence of random variables is coarser than the classification of easy (quantum) groups.

In this work, we classify all de Finetti-type theorems for $*$-random variables in classical and free probability. We observe that the maximal distributional symmetries that yield non-trivial de Finetti theorems in these settings are given by unitary groups and quantum unitary groups.  Therefore, the maximal distributional symmetries for de Finetti theorems of $*$-random variables arise from (quantum) groups lying between the (quantum) permutation group and the (quantum) unitary group.  

In \cite{TW1,TW2}, Tarrago and Weber classified unitary easy (quantum) groups in both the free and the group cases. Similar to the self-adjoint case, we will see that the maximal distributional symmetries for $*$-random variables are represented by certain easy groups or easy quantum groups as identified by Tarrago and Weber.  We will adopt similar notation from \cite{TW1,TW2} for quantum groups satisfying specific universal relations.

Based on the assumption that random variables are classically independent or freely independent, we have the following de Finetti-type theorems for finite sequences.

\begin{Theorem}\label{maximal}
Let  $(\M,\E)$ be a $\B$-valued probability space.
\begin{itemize}
\item[A.] Free case: Assume that  $x_1,\cdots,x_n\in \M$ are  freely independent. Then,

\begin{enumerate}
\item  $x_i$'s are identically distributed if and only if $(x_1,\cdots,x_n)$ is $C(\qs,n)$-invariant over $\B$.
\item  $x_i$'s are identically distributed \textit{orthogonal} elements if and only if $(x_1,\cdots,x_n)$ is $C(\qo,n)$-invariant over $\B$.
\item  $x_i$'s are identically distributed \textit{shifted orthogonal} elements if and only if $(x_1,\cdots,x_n)$ is $C(\qb_s,n)$-invariant over $\B$.
\item  $x_i$'s are identically distributed \textit{symmetric} elements if and only if $(x_1,\cdots,x_n)$ is $C(\qh_s,n)$-invariant over $\B$.
\item  $x_i$'s are identically distributed \textit{m-unitary} elements for $m\geq 3$ if and only if $(x_1,\cdots,x_n)$ is $C(\qh_m,n)$-invariant over $\B$.
\item  $x_i$'s are identically distributed \textit{free unitary} elements if and only if $(x_1,\cdots,x_n)$ is $C(\qh_0,n)$-invariant over $\B$.
\item  $x_i$'s are identically distributed \textit{R-diagonal} elements if and only if $(x_1,\cdots,x_n)$ is $C(\rqh,n)$-invariant over $\B$.
\item  $x_i$'s are identically distributed \textit{shifted circular} elements if and only if $(x_1,\cdots,x_n)$ is $C(\qb,n)$-invariant over $\B$.
\item $x_i$'s are identically distributed \textit{circular} elements if and only if $(x_1,\cdots,x_n)$ is $C(\qu,n)$-invariant over $\B$.
\end{enumerate}

\item[B.] Classical case: Assume that $\M$ is a commutative probability space and $x_1,\cdots,x_n\in \M$ are  conditionally independent.  Then,
\begin{enumerate}
\item  $x_i$'s are identically distributed if and only if $(x_1,\cdots,x_n)$ is $C(\cs,n)$-invariant over $\B$.
\item  $x_i$'s are identically distributed \textit{orthogonal} elements if and only if $(x_1,\cdots,x_n)$ is $C(\co,n)$-invariant over $\B$.
\item  $x_i$'s are identically distributed \textit{shifted orthogonal} elements if and only if $(x_1,\cdots,x_n)$ is $C(\cb_s,n)$-invariant over $\B$.
\item  $x_i$'s are identically distributed \textit{symmetric} elements if and only if $(x_1,\cdots,x_n)$ is $C(\ch_s,n)$-invariant over $\B$.
\item  $x_i$'s are identically distributed \textit{m-unitary} elements for $m\geq 3$ if and only if $(x_1,\cdots,x_n)$ is $C(\ch_m,n)$-invariant over $\B$.
\item  $x_i$'s are identically distributed \textit{unitary} elements if and only if $(x_1,\cdots,x_n)$ is $C(\ch_0,n)$-invariant over $\B$.
\item  $x_i$'s are identically distributed \textit{shifted Gaussian} elements if and only if $(x_1,\cdots,x_n)$ is  $C(\cb,n)$-invariant over $\B$.
\item $x_i$'s are identically distributed \textit{Gaussian} elements if and only if $(x_1,\cdots,x_n)$ is $C(\cu,n)$-invariant over $\B$.
\end{enumerate}
\end{itemize}
\end{Theorem}

On can see that the above theorem holds in a purely algebraic frame work.
 A difference between the free case and the classical case is that R-diagonal elements do not exist in the classical case. 
 
For infinite sequences of random variables, the distributional symmetries are given by a sequence of (quantum) groups $\A=(A_n)_{n\in \mathbb{N}}$ of size $n$.
For the sequence of quantum groups $C(\A,n)$ of size $n$ with the same property appeared in the previous theorem, we write $\A=\left( C(\A,n)\right)_{n\in \N}$ e.g, $\qs=\left( C(\qs,n)\right)_{n\in \N}$.
Two simple examples of exchangeable sequences of random variables are: (1) independent and identically distributed (i.i.d.) random variables, and (2) entirely identical random variables.  A rough idea for proving de Finetti-type theorems is to separate the i.i.d. part from the identical part of an exchangeable sequence. To analyze the identical part, one can either shift the indices of the random variables to infinity \cite{Ko} or take the limit of the averages of finite sequences associated with the Haar state on (quantum) groups \cite{Cu2}.  Both methods achieve the same goal: identifying the invariant tail algebra on which those (quantum) groups act trivially.  Consequently, the distributional symmetries pass to operator-valued random variables that are classically independent or freely independent.
Together with Curran's de Finetti type theorems on exchangeable $*$-random variables in $W^*$-probability space, we have the following results.

\begin{Theorem}\label{de Finetti}
Let $(\domain,\phi)$ be a $W^*$-probability space, $\phi$ is normal faithful and $\domain$ is generated by an infinite sequence of random variables $(x_i)_{i\in \N}$. 
Suppose that the joint distribution of $(x_i)$ is $\A$-invariant, where $\A$ is one of $\qs$, $\qo$, $\qb_s$, $\qh_s$, $\qb $, $\qh_m$, $\qh_0$, $\rqh$, $\qu$, $\cs$, $\co$, $\cb_s$, $\ch_s$, $\cb$, $\ch_m$, $\ch_0$, $\cu$. 
Then, there is a unital $W^*$-subalgebra $\range$ and a $\phi$-preserving conditional expectation $\E:\domain\rightarrow \range$ such that the following hold:
\begin{itemize}
\item[A.] Free case:

\begin{enumerate}
\item If $\A=\qs$,  then $(x_i)_{i\in \N}$ are freely independent in $(\domain,\E)$ with identical $\range$-valued distributions.
\item If $\A=\qo$,   then $(x_i)_{i\in \N}$ are freely independent in $(\domain,\E)$ with identical $\range$-valued orthogonal distributions.
\item If $\A=\qb_s$    then $(x_i)_{i\in \N}$ are freely independent in $(\domain,\E)$ with identical $\range$-valued shifted orthogonal distributions.
\item If $\A=\qh_s$,   then $(x_i)_{i\in \N}$ are freely independent in $(\domain,\E)$ with identical $\range$-valued distributions  such that $x_i,-x_i$ are identically distributed.
\item If $\A=\qb$    then $(x_i)_{i\in \N}$ are freely independent in $(\domain,\E)$ with identical $\range$-valued shifted circular distributions.
\item If $\A=\qh_m$, for $m\geq 3$, then $(x_i)_{i\in \N}$ are freely independent in $(\domain,\E)$ with identical $\range$-valued $m$-unitary distributions.
\item If $\A=\qh_0$, then $(x_i)_{i\in \N}$ are freely independent in $(\domain,\E)$ with identical $\range$-valued free unitary distributions.
\item If $\A=\rqh$,  then $(x_i)_{i\in \N}$ are freely independent in $(\domain,\E)$ with identical $\range$-valued R-diagonal distributions.
\item If $\A=\qu$,  then $(x_i)_{i\in \N}$ are freely independent in $(\domain,\E)$ with identical $\range$-valued circular distributions.
\end{enumerate}

\item[B.] Classical case:  Suppose that  $\domain$ is commutative.  In this case, $(x_i)_{i\in \N}$ are classical complex-valued random variables. 

\begin{enumerate}
\item If $\A=\cs$,  then $(x_i)_{i\in \N}$ are conditionally independent in $(\domain,\E)$ with identical distribution given $\range$.
\item If $\A=\co$,  then $(x_i)_{i\in \N}$ are conditionally independent in $(\domain,\E)$ with identical orthogonal distribution given $\range$.
\item If $\A=\cb_s$   then $(x_i)_{i\in \N}$ are conditionally independent in $(\domain,\E)$ with identical shifted orthogonal distribution given $\range$.
\item If $\A=\ch_s$   then $(x_i)_{i\in \N}$ are conditionally independent in $(\domain,\E)$ with identical distribution given $\range$ such that $x_i,-x_i$ are identically distributed
\item If $\A=\cb_1$    then $(x_i)_{i\in \N}$ are conditionally independent in $(\domain,\E)$ with identical shifted Gaussian distribution given $\range$.
\item If $\A=\ch_m$, for $m\geq 3$, then $(x_i)_{i\in \N}$ are conditionally independent in $(\domain,\E)$ with identical $m$-unitary distribution given $\range$.
\item If $\A=\ch_0$, then $(x_i)_{i\in \N}$ are conditionally independent in $(\domain,\E)$ with identical unitary distribution given $\range$.
\item If $\A=\cu$,  then $(x_i)_{i\in \N}$ are conditionally independent in $(\domain,\E)$ with identical Gaussian given $\range$.
\end{enumerate}
\end{itemize}
\end{Theorem}

We will see that Theorem 2 relies on transferring quantum symmetries the operator-valued joint distribution, which is based on the conditional expectation is $\phi$-preserving and $\phi$ is faithful. 
In fact, the conditional expectation exists in a more general case under the assumption that the GNS representation of $\phi $ is faithful, which includes more interesting examples, see \cite{DK,DKW}.

Let $X=(x_1,\cdots,)$ be an infinite sequence of random variables in a probability space with faithful state $\phi$. 
Then we can define its distributional symmetry set to be a family of sequences of compact quantum groups as follows
$$\mathcal{DS}(X,\phi)=\{A=(\A_n)_{n\in \N}| (x_1,\cdots,x_n) \,\text{is}\, \A_n-\text{invariant for all}\, n \}.$$  
On the other hand, given a sequence of compact quantum groups $A=(\A_n)_{n=1,\cdots}$ of dimension $n$, we can define $\A$-invariant family of  infinite sequences of random variables as follows
$$\mathcal{JD}(A)=\{ (X=(x_n)_{n\in\N},\phi)| (x_1,\cdots,x_n) \,\text{is}\, \A_n-\text{invariant for all}\, n\}.$$

We will see that finite exchangeable sequences of random variables from a probability space with a faithful state, whether they have non-zero expectation or zero covariance, are invariant under a maximal (quantum) group.
Therefore, for an infinite (quantum) exchangeable sequences of random variables $X=(x_n)_{n\in \N}$,  
$\mathcal{DS}(X,\phi)$ has a maximal elements $\A=(\A_n)$ in the sense that for all $\B=(\B_n)\in \mathcal{DS}(X,\phi)$, we have $\B_n\subset \A_n$ for all $n$.  For convenience, will write $\mathcal{DS}(X,\phi)=\A$.
 The  determine all those possible maximal elements $\A=(\A_n)$ for exchangeable commutative random variables and quantum exchangeable random variables.


\begin{Theorem}\label{all}
Let $(\M,\phi)$ be a $W^*$-probability space,$\phi$ is normal faithful and $\M$ is generated by an infinite sequence of non-zero random variables $(x_i)_{i\in \N}$. 
If the $*$-joint distribution of $X=(x_i)_{i\in \N}$ is $\A$-invariant, where $\A=(\A(n))_{n\in\N}$ be a sequence of  compact matrix quantum groups for $n$.  Then the following hold:

\begin{itemize}

\item[A.] Free case: If $\qs\subset\mathcal{DS}(X,\phi)$, then $\mathcal{DS}(X,\phi)$ takes value from $\qs$, $\qo$, $\qb_s$, $\qh_s$, $\qb $, $\qh_m$, $\qh_0$, $\rqh$, $\qu$.

\item[B.] Classical case:  If $\cs\subseteq \mathcal{DS}(X,\phi)$ and $\M$ is commutative, then $\mathcal{DS}(X,\phi)$ takes value from $\cs$, $\co$, $\cb_s$, $\ch_s$, $\cb$, $\ch_m$, $\ch_0$, $\cu$. 
\end{itemize}
\end{Theorem}

The above theorem shows that, for infinite sequences, finitely many types of distributions in free probability and classical probability can be characterized by de Finetti-type theorems. 
In addition, all the quantum and classical groups in the above theorem provide maximal distributional symmetry for the corresponding distributions.  

The following diagram illustrates the relationships between these maximal distributional symmetries in free probability. For classical probability, the diagram is almost identical, except that the quantum signs are removed, and $\rqh$ has no classical analogue.

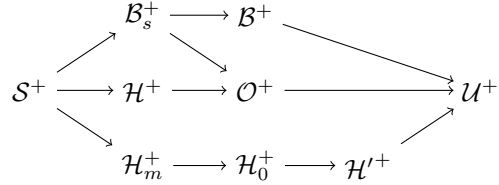
\begin{center}

\begin{tikzpicture}[xscale=1.5,yscale=1]

\path 
node   (m11) at (0,2) {$\qs$}         
node   (m21) at (1,3) {$\qb_s$}
node   (m22) at (1,2) {$\qh$}
node   (m23) at (1,1) {$\qh_m$}
node   (m33) at (2,1) {$\qh_0$}
node   (m32) at (2,2) {$\qo$}
node   (m31) at (2,3) {$\qb$}     
node   (m42) at (4,2) {$\qu$}
node   (m43) at (3,1) {$\rqh$}                               
 ;
{[->]       
\draw   (m11) --  (m21); 
\draw   (m11) --  (m22);
\draw   (m11) --  (m23); 
\draw   (m21) --  (m31);
\draw   (m21) --  (m32);
\draw   (m23) --  (m33);
\draw   (m32) --  (m42);
\draw   (m33) --  (m43);
\draw   (m31) --  (m42);
\draw   (m43) --  (m42);
\draw   (m22) --  (m32);
      }
   
\end{tikzpicture}

\end{center}

It is worth mentioning that in \cite{BCFMW}, the authors used the unitary dual group—where the coproduct map for quantum groups to tensor products is replaced by a map to free products—to define distributional symmetries. The unitary dual group is much larger than Wang's quantum unitary groups. Consequently, even for finite sequences, invariance under such a symmetry implies freeness.

Besides the first introductory section, the rest of the paper is organized as follows. In Section 2, we recall necessary notions and results from classical probability, free probability, and unitary quantum groups.

In Section 3, we demonstrate that all finite sequences of random variables admit a maximal distributional symmetry arising from a quantum semigroup. Moreover, if the random variables have non-zero expectation or zero covariance, then the maximal distributional symmetry is given by a quantum unitary group.

n Section 4, we establish the universal conditions for groups and quantum groups appearing in our de Finetti-type theorems. We will see that these universal conditions define universal $C^*$-bialgebras, which naturally become compact quantum groups.  Additionally, concrete corepresentations will be provided to demonstrate that all these universal quantum groups are distinct.

In Section 5, we introduce random variables with vanishing cumulants and provide examples in the scalar case.

In Section 6, we prove Theorem 1 and Theorem 2.  

In Section 7, we explore the properties for  universal relations of quantum groups to establish Theorem 3.  

In Section 8, we prove Theorem 3.

Section 9 contains concluding comments and remarks.

\section{Definitions and Notation}
We begin by introducing the fundamental notions of operator-valued free probability theory, with a particular focus on its combinatorial aspects as extensively developed in \cite{Sp1}. In the scalar-valued case, the definitions and results can be adapted by replacing the conditional expectations with linear functionals, see\cite{NS, VDN}. 

\begin{definition} A \textit{$\B$-valued probability space} $(\domain,\E)$ consists of a unital $*$-algebra $\A$, a $*$-subalgebra $\B$ of $\A$ contains $1_\A$ and a unital $\B-\B$ bimodule linear map $\E:\A\rightarrow \B$, i.e.
 $\E[1_\A]=1_\A$,  $\E[a_1+a_2]=\E[a_1]+\E[a_2]$
and   $$\E[b_1a_1b_2]=b_1\E[a_1]b_2,$$
  for all $a_1, a_2\in \A$, $b_1,b_2\in \B$.  Elements in $\domain$ will be called \textit{$\range$-valued random variables}, or simply \textit{random variables}. 
  
\end{definition}
When $\B=\C$, $(\A,\E)$ is simply referred to as a noncommutative probability space. 
Furthermore, $(\A,\E)$  is a $W^*$-probability space if $\domain$ is a von Neumann algebra and $\E$ is a normal state on $\domain$. 
$\E$ is faithful if $\E[xx^*]=0$ implies $x=0$. 
In this paper, we will assume $\E$ to be faithful but not necessarily tracial, i.e., $\E[xy]\neq\E[yx]$ in general. We say that $(\A,\E)$ is a commutative probability space if $\A$ is a commutative algebra.

The following examples illustrate how the preceding definitions relate to classical probability.
\begin{example} Let $\Omega, \Sigma, \mu)$ be a classical probability space,  and let  $\F\subset \Sigma$ be a $\sigma$-subalgebra.  
\begin{enumerate}
\item The pair $(L^\infty(\mu), \EE)$ is a commutative probability space, where $L^\infty(\mu)$ is the algebra of bounded $\Sigma$-measure functions, and $\EE$ is the expectation functional $\EE(f)=\int f d\mu$.
\item For $1\leq p\leq \infty$,  let $L^p(\mu)$ be the set of random variables with finite $p$-th moment. Then, $L(\mu)=\bigcap\limits_{p\in \N} L^p(\mu)$ is an algebra and  $(L(\mu),\EE)$ is a commutative probability space.
\item Let $\domain=L^\infty(\mu)$, and let $\range=L^\infty(\mu|_\F)$ be the subalgebra of bounded $\F$-measurable function on $\Omega$. Then, $(\domain, \EE[\cdot|\F])$ is a $\range$-valued probability space.
\end{enumerate}
\end{example}

In the real  noncommutative case, i.e., $\domain$ is a noncommutative algebra, to study the products of random variables and constants from $\range$, we will need to take account of the order of the product.  For instance,  $xyb$, $xby$, $ybx$, $ybx^*$ are different elements from $\domain$ in general.  
In addition, our  noncommutative random variables have finite mixed moments of orders since the expectation is defined on the algebra generated them.
To study the elements in $\domain$, we introduce  the following $*$-algebra of noncommutative polynomials with coefficients in $\range$.

\begin{notation} 
Given an index set $\I$,  we denote by  $\Q_{\I}(\range)=\range\langle t_i, t_i^*: i\in \I\rangle$ the $*$-algebra freely generated by $\range$, indeterminants $\{t_i|i\in \I\}$ and their adjoints $\{t^*_i|i\in \I\}$.  
We set $\Q_{n}(\range)=\Q_{\I}(\range)$ if $\I=\{1,\cdots,n\}$  and  set  $\Q_\I=\Q_{\I}(\range)$ if $\range=\C$.

By the universality of $\Q_{\I}(\range)$, for any family of random variables $x=(x_i)_{i\in \I}$ in a $\range$-valued probability space $(\domain, \E)$, 
there exists a unique $*$-homomorphism $ev_x: \Q_{\I}(\range)\rightarrow \domain$ such that $ev_x(t_i)=x_i$ and $ev_x(b)=b$ for all $b\in \range.$ 
We will denote $ev_x(p)$ by $p(x)$, for $p\in \Q_{\I}(\range)$.
\end{notation}

\begin{definition}  Let $x=(x_i)_{i\in \I}$ be a family of random variables in a $\range$-valued probability space $(\domain, \E)$.  The \textit{joint distribution} of $x$ is a map  $\E_x: \Q_{\I}(\range)\rightarrow \range$ defined by
$$\E_x(p)=\E(p(x)).$$
When $\I=\{ i\}$ contains a single element, $\E_{x_i}$ is called the \textit{distribution} of $x_i$.
\end{definition}

Note that $\E_x$ is $\B-\B$ bimodule linear since $\E$ is.   Consequently, the joint distribution of $(x_i)_{i\in \I}$ is determined by their joint moments
$$\E[b_0x_{i_1}^{d_1}b_1x_{i_2}^{d_2}\cdots b_{k-1}x_{i_k}^{d_k}b_k],$$
for $k\in \N$, $b_0,\cdots,b_{k}\in\range$, $i_1,\cdots,i_k\in \I$ and $d_1,\cdots,d_k\in \{1, *\}$.

\subsection{Free probability} Now, we turn to the introduction of independence relations in $\range$-valued probability.

\begin{definition}  Let $x=(x_i)_{i\in \I}$ be a family of random variables in a $\range$-valued probability space $(\domain, \E)$, and let $\E_x$ be the joint distribution of $x$.
\begin{enumerate}
\item The variables are called \textit{conditionally independent with respect to $\E$} if 
$$ \E_x[p_1\cdots p_k]=\E_x[p_1]\cdots \E_x[p_k],$$
whenever $p_j\in\range\langle t_{i_j}, t_{i_j}^*\rangle$ for $j=1,\cdots,k$ and $i_1,\cdots,i_k\in \I$ are distinct.
\item The variables are called \textit{free with amalgamation over $\range$} or  \textit{free with respect to $\E$} if 
$$ \E_x[p_1\cdots p_k]=0,$$
whenever $p_j\in\range\langle t_{i_j}, t_{i_j}^*\rangle$ such that $\E_x[p_j]=0$ for $j=1,\cdots,k$  and $i_1,\cdots,i_k\in \I$ are such that $i_1\neq i_2 \neq \cdots \neq i_k$.
\end{enumerate}
\end{definition}

\begin{remark}
Conditional independence, in general, is not universal in the sense of Speicher \cite{Sp}, because the joint distribution $\E_x$ is not necessarily determined solely by the distributions of the $\E_{x_i}$, even when $\range = \C$. For example, the conditional independence relation provides no information about the joint moment $\E[x_1 x_2 x_1]$. Indeed, there are infinitely many ways to define $\E_x$ based on the conditional independence assumption \cite{JL}.  

However, if we further assume that $\domain$ is a commutative algebra, then $\E_x$ is completely determined by the distributions of the $\E_{x_i}$, since products of random variables can be reordered freely. In this case, we recover the classical notion of conditional independence in probability.

By contrast, free independence provides a recursive approach to determine the joint distribution of $(x_i)_{i \in \I}$ from the distribution of each $x_i$, by reducing the degree of noncommutative polynomials step-by-step.

Additionally, if $\domain$ is commutative and $(x_i)_{i \in \I}$ are free in $\domain$, then at most one $x_i$ lies outside $\range$, and the others are from $\range$.
\end{remark}

\subsection{Partitions} Both operator-valued conditional independence in commutative probability spaces and free independence can be characterized by combinatorial theories developed by Speicher \cite{Sp1}; see also \cite{NS}. These theories mainly rely on the concepts of noncrossing partitions and cumulants, which we will introduce in the following.

\begin{definition}\label{partition} Let $\s$ be an ordered set.
\begin{enumerate}
\item A \textit{partition} $\pi$ of $\s$ is a collection of disjoint, nonempty sets $V_1,\cdots,V_r$ whose union is $\s$. 
The sets $V_1,\cdots,V_r$ are called blocks of $\pi$. 
The collection of all partitions of $S$ will be denoted by $\pp(\s)$.
\item Given two partitions $\pi$, $\sigma$, we say $\pi\leq \sigma$ if each block of $\pi$ is contained in a block of $\sigma$.
\item A partition $\pi\in \pp(\s)$ is \textit{noncrossing} if there is no quadruple $(s_1,s_2,r_1,r_2)$ such that $s_1<r_1<s_2<r_2$, $s_1,s_2\in V$, $r_1,r_2\in W$ and $V,W$ are two different blocks of $\pi$. The family of noncrossing partitions on $\s$ is denoted by $NC(\s).$


\item Let $\I$ be an index set, ${\ii}=(i_1,\cdots,i_k)\in \I^k$ be a sequence of indices from $\I$ and $[k]=\{1,\cdots,k\}$ be set of $k$-integers with the natural order. We denote by $\ker\ii$ the element of $\pp([k])$ whose blocks are the equivalence classes of the relation
$$s\sim t\Leftrightarrow i_s=i_t$$

\end{enumerate}
\end{definition}

To characterize specific distributions of $*$-random variables that will appear in this paper, we introduce the following definitions associated with sequences of $\{1,*\}$.
\begin{definition} Let $[k]=\{1,\cdots,k\}$ and $\dd=(d_1,\cdots,d_k)$ be a sequence of k elements from $\{1,*\}$. Thus, partitions from $NC(k)$ can be used to partition $\dd$
\begin{enumerate}
\item If $ \pi\in NC(k)$ such that, for every block of $\pi$, the difference of the number of $*$ and $1$ of  is divided by $m$, then we say that $\pi$ is $m$-divisible with respect to $\dd$. 
The set of $m$-divisible noncrossing partitions with respect to $\dd$ is denoted by $NC^{\dd,m}(k)$.

\item If $\pi \in NC(k)$ such that, for every block of $\pi$, the number of $*$ and $1$ are equal, then we say that $\pi$ is $\infty$-divisible with respect to $\dd$. The set of $\infty$-divisible noncrossing partitions with respect to $\dd$ is denoted by $NC^{\dd,\infty}(k)$.

\item If $ \pi\in NC(k)$ such that  $*$ and $1$ appear alternately with equal times in  every block of $\pi$, then we say that $\pi$ is alternating with respect to $\dd$.  
The set of alternating partitions with respect to $\dd$ is denoted by $NC^{\dd}(k)$.

\item  If $ \pi\in NC_2(k)$ such that every block of $\pi$ contains $*$ and $1$, then we say that $\pi$ is an alternating pair partition with respect to $\dd$.  
The set of alternating pair partitions with respect to $\dd$ is denoted by $NC_2^{\dd}(k)$.
\end{enumerate}
\end{definition}

\begin{remark}
Given integers $m_1, m_2$ such that $m_1$ divides $m_2$, for all $k\in \N$ and $\dd\in \{1,*\}^k$, we have 
$$ NC_2^{\dd}(k)\subseteq NC^{\dd}(k)\subseteq NC^{\dd,\infty}(k) \subseteq NC^{\dd,m_2}(k) \subseteq NC^{\dd,m_1 }(k) $$
\end{remark}

Similarly, we can define corresponding partitions of $\dd$ for studying distributions of commuting random variables.

\subsection{$\range$-functionals, cumulants and moments}

\begin{definition} Let $(\domain,\E)$ be a $\range$-valued probability space.
\begin{enumerate}

\item A $\range$-valued functional is a $k$-linear map $\rho^{(k)}:\domain^k\rightarrow \B$  such that
$$\rho^{(k)}(b_0a_1b_1,a_2b_2,\cdots,a_kb_k)=b_0\rho(a_1,b_1a_2,\cdots,b_{k-1}a_k)b_k$$
for all $a_1,\cdots,a_k\in\domain$ and $b_0,\cdots,b_k\in \range$.

\item If  $\A$ is commutative. The sequence of $\B$-valued $k$-linear maps $(\rho^{(k)})_{k\in\N}$  can be extended to $\B$-valued linear functionals $\rho^{(\pi)}$ with $\pi\in P(k)$ as following:
$$\rho^{(\pi)}(a_1,\cdots,a_n)=\prod\limits_{V\in \pi} \rho{(V)}(a_1,\cdots,a_k),$$
where if $V=(i_1<i_2<\cdots<i_s)$ is a block of $\pi$ then 
$$\rho(V)(a_1,\cdots,a_n)=\rho^{(s)}(a_{i_1},\cdots,a_{i_s}). $$

\item Given a sequence of $\B$-valued $k$-linear maps $(\rho^{(k)})_{k\in\N}$,  for $\pi\in NC(k)$, we can define  $\rho^{(\pi)}: \Ak\rightarrow \B$ recursively as follows:
$$\rho^{(\pi)}(a_1,\cdots,a_k)=\rho^{(\pi\setminus V)}(a_1,\cdots,a_l\rho^{(s)}(a_{l+1,\cdots,a_{l+s}}),a_{l+s+1},\cdots,a_k)$$
where $V=(l+1,l+2,\cdots,l+s)$ is an interval block of $\pi$.

\item For each $k\in \N$, the \textit{$\B$-valued moment functions} $ E^{(k)}: \Ak\rightarrow \B $ is defined by $$ E^{(k)}(a_1,\cdots,a_k)=\E[a_1\cdots a_k].$$

\item The $\B$-valued free cumulant functions $\kappa^{(k)}:\Ak\rightarrow\B$ are defined recursively by the following:
$$E^{(k)}(a_1,\cdots,a_n)=\sum\limits_{\pi\in NC(k)}\kappa^{(\pi)}(a_1,\cdots,a_k) $$
\end{enumerate}
\end{definition}

According to Speicher \cite{Sp1}, the following vanishing conditions on cumulants characterize freeness and classical independence.

\begin{theorem}
Let $(\domain,\E)$ be a $\range$-valued probability space and $(x_i)_{i\in I}$ be a family of random variables in $\domain$. 
\begin{enumerate}
\item Suppose that $\domain$ is a commutative algebra.  $(x_i)_{i\in I}$ is conditional independent given $\range$   if 
$$\rho^{(k)}(b_0x_{i_1}b_1,x_2,\cdots,x_{i_k}b_k)=0,$$
for all $k\in\N$, whenever $i_{j_1}\neq i_{j_2}$ for some $1\leq j_1, j_2\leq k$.

\item  $(x_i)_{i\in I}$ is free over $\range$   if 
$$\kappa^{(k)}(b_0x_{i_1}b_1,x_2,\cdots,x_{i_k}b_k)=0,$$
for all $k\in\N$, whenever $i_{j_1}\neq i_{j_2}$ for some $1\leq j_1, j_2\leq k$.

\end{enumerate}

\end{theorem}

\subsection{Compact  quantum groups}

The notion of compact matrix quantum groups was introduced by Woronowicz in \cite{Wo2} and later refined to the following equivalent definition in \cite{Wo3}.

\begin{definition} \label{quantum groups}A compact matrix quantum group for $n\geq 1$ is a unital $C^*$-algebra $\A$ with a $C^*$-homomorphism $\Delta$ from $\A$ to $\A\otimes_{\min}\A$ such that
\begin{itemize}
\item $\A$ is generated by $n^2$ elements $u_{i,j}$, $i,j=1,\cdots,n$
\item the matrices $u=(u_{i,j})$ and $\bar{u}=(u_{i,j}^*)$ are invertible in $M_n(\A)$
\item $\Delta(u_{i,j})=\sum\limits_{k=1}^n u_{i,k}\otimes u_{k,j}$.
 \end{itemize}
A compact quantum group is unitary if  $u=(u_{i,j})$, and consequently  $\bar{u}_{i,j}=(u_{i,j}^*)$, is unitaries in $M_n(\A)$.
\end{definition}
Let $u^t=(u_{j,i})$ be the transpose of $u$. 
Since $\bar{u}$ is invertible, so is its conjugate transpose $(\bar u)^*=u^t$.

In \cite{Banica}, $\A$ is regarded as the set of continuous functions on a compact quantum Lie group $G$, denoted by $C(G)$, even though $\A$  may not be  commutative in general. 
The $\A$-valued matrix $u=(u_{i,j})$ is is referred to as  a corepresentation of $\A$. 
We denote the pair consisting of the quantum group $G$ and its corepresentation $u$ by $ (G,u)$.
A quantum subgroup of $(H,v)$ is defined as a pair $(G,u)$ such that there exists a $*$-homomorphism $\Phi: C(H) \rightarrow C(G)$ satisfying  $\Phi(v_{i,j})=u_{i,j}, \forall i,j=1,\cdots,n.$ 
In this context, we consider the maximal compact quantum unitary group $C(\qu,n)$ introduced by Wang in \cite{Wan2}. This quantum group is the universal $C^*$-algebra generated by $\{u_{i,j}|i,j=1,\cdots,n\}$ , satisfying the following algebraic relations:
$$\sum_{k=1}^n u_{i, k} u^*_{j, k}=\sum_{k=1}^n u^*_{i, k} u_{j, k} =\sum_{k=1}^n u_{k,i} u^*_{k,j} =\sum_{k=1}^n u_{k,i} u^*_{k,j}  =\delta_{i, j}1.$$
The coproduct is well-defined since $\left(\Delta(u_{i,j})\right)$ satisfies the above relations. 
In fact, we will show that requiring only two of the four summations to hold is sufficient; the remaining two will then follow automatically.

In this paper, we will study the distributional symmetries for de Finetti-type theorems in free probability, which lie between the following two fundamental quantum groups introduced by Wang in \cite{Wan2}.
\begin{definition} Quantum Unitary groups and quantum permutation groups:

\begin{enumerate} 
\item $C(\qu,n)$ is the universal $C^*$-algebra generated by $n^2$  elements $u_{i,j}$, such that $u=\left(u_{i j}\right)$ is a unitary in $M_n\left(C(\qu,n) \right)$.

\item $C(\qs,n)$ is the universal $C^*$-algebra generated by $n^2$ orthogonal projections $u_{i j}$, such that the sum along any row or column of $u=\left(u_{i j}\right) M_n\left(C(\qs,n)\right)$ is the identity of $C(\qs,n)$.
\end{enumerate}
\end{definition}


\subsection{Quantum distributional symmetries for $*$-random variables}
Notice that given a corepresentation $u=(u_{i,j})$ of a quantum group $\A$ for $n$, there exists a unique $*$- homomorphism $\alpha_\A: \B\langle t_k,t_k^*|1\leq k\leq n \rangle \rightarrow \B\langle t_k,t_k^*|k\in \N \rangle\otimes \A $,
which maps elements from the algebra generated by $t_k$, its adjoint $t_k^*$ for $1\leq k\leq n$ and an algebra $\B$ to its algebraic tensor product with $\A$ such that 
$$
\alpha_\A\left(t_i^d\right)= \sum_{k=1}^n t_k^d \otimes u_{k, i}^d 
$$
for any $1\leq k\leq n, d \in\{1, *\}$ and $\alpha_\A(b)=b\otimes 1_\A$ for all $b\in \range$.
Now, we will apply this coaction to the distributional symmetry induced by $ \A $ on $n$  noncommutative random variables.

\begin{definition}
Let $(x_1,\cdots,x_n)$ be a sequence of random variables in $(\domain, \E)$ and $\A$ is a compact quantum group generated by $u=(u_{i,j})$. We say that $(x_1,\cdots,x_n)$ is $\A$-invariant if 
$$
\left(\mu_x \otimes id_\A\right)\left(\alpha_\A(p)\right)=\mu_x(p) 1_\A
$$
for all $ p \in \B\left\langle t_k, t_k \mid k=1, \cdots, n\right\rangle$, where $\mu_x$ is the joint distribution of $\left(x_1, \cdots, x_n\right)$.
\end{definition}

More explicitly, a sequence of random variables $\left(x_1, \cdots, x_n\right)\in (\M,\E)$  is $\A$-invariant if 
$$
 \sum_{1 \leq i_1, \cdots i_k \leq n} \E\left[b_0 x_{i_1} b_1 x_{i_2} \cdots b_{k-1} x_{i_k} b_k\right] \otimes u_{i_1, j_1} \cdots u_{i_k, j_k}  =\E\left[b_0 x_{j_1} b_1 \cdots b_{k-1} x_{j_k} b_k\right]\otimes 1_\A,
$$
for all $1 \leqslant j_1, \cdots, j_k \leqslant n, b_0, \cdots, b_k \in \range$.

The distributional symmetry for an infinite sequence of random variables $(x_n)_{n\geq 1}$ characterized  by a sequence of compact quantum groups for $n\geq 1$. 
Given a corepresentation $u=(u_{i,j})$ of a quantum group $\A$ for $n$, the $*$- homomorphism $\alpha_\A$ naturally extends to $\tilde{\alpha}_\A: \B\langle t_k,t_k^*|k\in \N \rangle \rightarrow \B\langle t_k,t_k^*|k\in \N \rangle\otimes \A $ such that 

$$
\tilde{\alpha}_\A\left(t_i^d\right)= \begin{cases}\sum_{k=1}^n t_k^d \otimes u_{k, i}^d & \text { if } 1 \leqslant i \leqslant n \\ t_i^d & \text { otherwise }\end{cases}
$$
for any $k \in \mathbb{N}, d \in\{1, *\}$ and $\alpha_\A(b)=b\otimes 1_\A$ for all $b\in \range$.
For convenience, we will denote $\alpha_\A$ simply for $ \tilde{\alpha}_\A$.

\begin{definition}
Let $x=(x_n)_{n\geq 1}$ be an infinite sequence of random variables in $(\domain, \E)$ and $\AAA=(\A_n)_{n\geq 1}$ be sequence of compact quantum group for $n$. We say that $x$ is $\AAA$-invariant if 
$$
\left(\mu_x \otimes id_{\A_n}\right)\left(\alpha_{\A_n}(p)\right)=\mu_x(p) 1_{\A_n}
$$
for all $ p \in \B\left\langle t_k, t_k \mid k\geq 1 \right\rangle$, where $\mu_x$ is the joint distribution of $\left(x_1, \cdots, x_n\right)$.
\end{definition}

\section{Maximal distributional symmetry for finite sequences of random variables}

First, we generalize the concept of distributional symmetries arising from quantum groups to the distributional symmetries associated with a given finite sequence of random variables.

\begin{definition} 
Let $(x_i)_{i=1,\cdots,n}$ be a sequence of random variables in a $\B$-valued probability space $(\M, \E)$ and let $\mu$ be the joint distribution of $(x_i)_{i=1,\cdots,n}$ with respect to $\E$.
Let $\A$ be a unital $C^*$-algebra generated by $\{u_{i,j}|{i,=1,\cdots, n}\}$, and define the homomorphism
$\alpha:\B\langle X_1,\cdots,X_n\rangle \rightarrow  \B\langle X_1,\cdots,X_n\rangle\otimes \A$  generated by
$$\alpha(X_i)=\sum\limits_{k=1}^n  X_k \otimes u_{k,i}, \quad \alpha(b)=b \otimes 1_\A $$
for all $i=1,\cdots,n$ and $b\in \B$.
We say that $(x_i)_{i=1,\cdots,n}$ is 
$\A$-invariant if 
$$\mu [p]\otimes 1_\A =\mu \otimes id_{\A} (\alpha(p)),$$  for all $p\in \B\langle X_1,\cdots,X_n\rangle$.
\end{definition}

\begin{proposition}\label{x-invariant algebra}
Given a sequence of random variables  $(x_i)_{i=1,\cdots,n}$ in a $\B$-valued probability space $(\M, \E)$,  
let $\A$ be the universal $C^*$-algebra  generated by $\{u_{i,j}|{i,=1,\cdots, n}\}$ such that the matrix
$u=(u_{i,j})$ is unitary in $\A\otimes M_n(\C)$ and $(x_i)_{i=1,\cdots,n}$ is  $\A$-invariant.
Then, $\A$ forms a compact quantum group.
\end{proposition}
\begin{proof}
To establish that $\A$ forms a compact quantum group, we need to show that the corepresentation $u = (u_{i,j})$ is non-empty and that the map $\Delta(u_{i,j}) = \sum_{k=1}^n u_{i,k} \otimes u_{k,j}$ defines a coproduct on $\A$.  

Since $\A$ contains the trivial representation obtained by sending $u_{i,j}$ to $\delta_{i,j}$, $u = (u_{i,j})$ is non-empty.

Secondly, following the definition of Wang's quantum unitary groups, $(\sum\limits_{k=1}^n u_{i,k}\otimes u_{k,j} )_{i,j=1,\cdots, n}$  is obviously unitary in $\A\otimes M_n(\C)$ .

Let $(d_1,\cdots,d_m)$ be a sequence of $m$ elements from $\{1,*\}$.  Then, we have 
$$
\begin{aligned}
&b_0 x^{d_1}_{j_1} b_1 x^{d_2}_{j_2} \cdots x^{d_m}_{j_m} b_m \otimes (\sum\limits_{k_1=1}^n u_{j_1,k_1}\otimes u_{k_1,i_1})^{d_1}\cdots(\sum\limits_{k_1=1}^n u_{j_m,k_m}\otimes u_{k_m,i_m})^{d_m}\\
=&\sum\limits_{k_1,\cdots,k_m=1}^n  b_0 x^{d_1}_{j_1} b_1 x^{d_2}_{j_2} \cdots x^{d_m}_{j_m} b_m \otimes \left( u^{d_1}_{j_1,k_1}\cdots u^{d_m}_{j_m,k_m}\right)\otimes \left( u^{d_1}_{k_1,i_1}\cdots u^{d_m}_{k_m,i_m}\right)
\end{aligned}
$$
and 
$$
\begin{aligned}
 & \mathbb{E}\otimes Id_{\A\otimes \A} \left[   \sum\limits_{j_1,\cdots,j_m, k_1,\cdots,k_m=1}^n  b_0 x^{d_1}_{j_1} b_1 x^{d_2}_{j_2} \cdots x^{d_m}_{j_m} b_m \otimes \left( u^{d_1}_{j_1,k_1}\cdots u^{d_m}_{j_m,k_m}\right)\otimes \left( u^{d_1}_{k_1,i_1}\cdots u^{d_m}_{k_m,i_m}\right) \right]\\
= & \mathbb{E}\otimes Id_{\A\otimes \A} \left[   \sum\limits_{ k_1,\cdots,k_m=1}^n  b_0 x^{d_1}_{k_1} b_1 x^{d_2}_{k_2} \cdots x^{d_m}_{k_m} b_m \otimes 1_{\A}\otimes \left( u^{d_1}_{k_1,i_1}\cdots u^{d_m}_{k_m,i_m}\right)   \right]\\
= & \mathbb{E}\otimes Id_{\A\otimes \A} \left[   \sum\limits_{ k_1,\cdots,k_m=1}^n  b_0 x^{d_1}_{i_1} b_1 x^{d_2}_{i_2} \cdots x^{d_m}_{i_m} b_m \otimes 1_{\A}\otimes 1_{\A}  \right].\\
\end{aligned}
$$

Therefore, the family $\{\sum\limits_{k=1}^n u_{i,k}\otimes u_{k,j} \}$ is $(x_i)_{i=1,\cdots,n}$-invariant. By the universality of $\A$, this family induces a $*$-homomorphism $\Delta: \A \to \A \otimes_{\min} \A$ defined by
$$
\Delta(u_{i,j}) = \sum_{k=1}^n u_{i,k} \otimes u_{k,j},
$$
which shows that $\Delta$ is a coproduct on $\A$.
\end{proof}

Recall that a quantum semigroup is a $C^*$-algebra $\mathcal{A}$ equipped with a $*$-homomorphism $\Delta \in \operatorname{Mor}(\mathcal{A}, \mathcal{A} \otimes \mathcal{A})$, known as comultiplication, such that
$$
\left(\Delta \otimes i d_{\mathcal{A}}\right) \Delta=\left(i d_{\mathcal{A}} \otimes \Delta\right) \Delta .
$$
Once the universal condition that $u = (u_{i,j})$ is unitary, as specified in Proposition \ref{x-invariant algebra}, is removed, the resulting universal $C^*$-algebra forms a quantum semigroup.

In the following, we will show that the condition for $u=(u_{i,j})$ to be unitary in $\A\otimes M_n(\C)$ automatically holds in certain cases, including many interesting and well-known examples arising from free probability.

\begin{definition}
A sequence $(x_i)_{i=1,\cdots,n}$ in a $\range$-valued probability space  $(\M, \E)$ is 2-exchangeable if their first and second moments are exchangeable. That is
$$\E[x_i]=\E[x_{\sigma(i)}],\quad \E[x_ib x_j]=\E[x_{\sigma(i)}b x_{\sigma(j)}]$$
for all $\sigma\in S_n$, $1\leq i,j\leq n$, $b\in\range$.
\end{definition}
In the language of joint distribution for noncommutative random variables, $(x_i)_{i=1,\cdots,n}$ in a $\range$-valued probability space  $(\M, \E)$ is 2-exchangeable if $$\mu [p]\otimes 1_{C(S,n)}=\mu \otimes id_{C(S,n)} (\alpha(p)),$$  for all $p\in \B\langle X_1,\cdots,X_n\rangle$ such the $\deg p\leq 2$.

\begin{lemma}\label{Unitary in M(A)}
Let $\A$ be a unital $C^*$-algebra, and $\{u_{i,j}|{i,=1,\cdots, n}\}\subset \A$ such that 
$$\sum\limits_{k=1}^n  u_{k,i}u^*_{k,j}=\sum\limits_{k=1}^n  u^*_{k,i}u_{k,j}=\delta_{i,j }1_\A$$ for all $i,j$.
Then, $u=(u_{i,j})$ is unitary in $ M_n(\A)$.
 \end{lemma}
\begin{proof}

Notice that  $u^*u=1_{M_n(\A)}$,  $u$ is a partial isometry in $M_n(\A)$. If $uu^*< 1_{M_n(\A)}$, then the sum of the diagonal elements of  $ uu^*$ is less than $n1_\A$. However,  the sum of the diagonal elements of  $ uu^*$ is given by
$$\sum\limits_{i=1}^n \sum\limits_{j=1}^n u_{i,j}u_{i,j}^* =\sum\limits_{j=1}^n \sum\limits_{i=1}^n  u_{i,j}u_{i,j}^*=n1_\A.$$
Therefore, $u=(u_{i,j})$ is unitary in $ M_n(\A)$.

\end{proof}

\begin{proposition}\label{unitary condition}
Let $(x_i)_{i=1,\cdots,n}$ be a sequence of 2-exchangeable random variables in a $\B$-valued probability space $(\M, \E)$  and $\A$ be a unital   $C^*$-algebra generated by $\{u_{i,j}|{i,=1,\cdots, n}\}$ such that $(x_i)_{i=1,\cdots,n}$ is $\A$-invariant. 
Suppose  that  one of the following conditions holds:
\begin{enumerate}
\item  $\E[x_1]\neq 0$.
\item $\E[x_1]=0$ and $\E[x_1bx_2]=0$ for all $b\in \B$.
\end{enumerate}
Assume further that $(x_i)_{i=1,\cdots,n}$ are nonzero and not identical such that for $i\neq j$, there exists $b\in \B$ such that $ \E[x_ibx_i^*]\neq \E[x_ibx_j^*] $. Then, the matrix $u=(u_{i,j})$ is unitary in $M_n(\A)$.
\end{proposition}

\begin{proof}
Case (1):   
Since  $\E\left[x_i\right] \otimes 1_A=\sum\limits_{k=1}^n \E\left[x_k\right] \otimes u_{k, i}$ and   $\E[x_i]=\E[x_j]\neq 0$ for all $i\neq j$,
we have 
$\sum\limits_{k=1}^n  u_{k, i}=1_\A$ for  all $i=1,..,n$. It follows that $\sum\limits_{k,l=1}^n  u_{k,i}u^*_{l,j}=(\sum\limits_{k=1}^n  u_{k,i} )(\sum\limits_{l=1}^n u^*_{l,j}) =1_\A,$ for all $i,j=1,\cdots,n$.

Notice that 
$$
\begin{aligned}
  \E[x_ibx_i^*]\otimes 1_\A&=\sum\limits_{k,l=1}^n \E[x_kbx_l^*]\otimes u_{k,i}u^*_{l,i}\\
  &=\sum\limits_{k=1}^n \E[x_kbx_k^*]\otimes u_{k,i}u^*_{k,i}+\sum\limits_{k,l=1,k\neq l}^n \E[x_kbx_l^*]\otimes u_{k,i}u^*_{l,i}\\
  &=\sum\limits_{k=1}^n \E[x_1bx_1^*]\otimes u_{k,i}u^*_{k,i}+\sum\limits_{k,l=1,k\neq l}^n \E[x_1bx_2^*]\otimes u_{k,i}u^*_{l,i}\\ 
 \end{aligned}
$$
The last equality holds because  $(x_i)_{i=1,\cdots,n}$ is 2-exchangeable. Thus,  we have 
$$ 
 \begin{aligned}  
 \E[x_ibx_i^*]\otimes 1_\A &=\E[x_1bx_1^*]\otimes  \sum\limits_{k=1}^n  u_{k,i}u^*_{k,i}+\E[x_1bx_2^*]\otimes (\sum\limits_{k,l=1}^n  u_{k,i}u^*_{l,i}-\sum\limits_{k=1}^n  u_{k,i}u^*_{k,i})\\  
  &=\E[x_1bx_1^*]\otimes  \sum\limits_{k=1}^n  u_{k,i}u^*_{k,i}+\E[x_1bx_2^*]\otimes (1_\A-\sum\limits_{k=1}^n  u_{k,i}u^*_{k,i}),\\  
\end{aligned}
$$
Therefore, we have 
$$ \E[x_1bx_1^*]\otimes (1_\A-\sum\limits_{k=1}^n  u_{k,i}u^*_{k,i})=\E[x_1bx_2^*]\otimes (1_\A-\sum\limits_{k=1}^n  u_{k,i}u^*_{k,i})$$
for all $b\in \B.$  
Recall that there exists $b\in \B$ such that $ \E[x_1bx_1^*]\neq \E[x_1bx_2^*].$
This implies that $\sum\limits_{k=1}^n  u_{k,i}u^*_{k,i}=1_\A$ for all $i$. Similarly, we have $\sum\limits_{k=1}^n  u^*_{k,i}u_{k,i}=1_\A.$ 

On the other hand, for $i\neq j$,  we have 
$$
\begin{aligned}
  \E[x_ibx_j^*]\otimes 1_\A&=\sum\limits_{k,l=1}^n \E[x_kbx_l^*]\otimes u_{k,i}u^*_{l,j}\\
  &=\sum\limits_{k=1}^n \E[x_kbx_k^*]\otimes u_{k,i}u^*_{k,j}+\sum\limits_{k,l=1,k\neq l}^n \E[x_kbx_l^*]\otimes u_{k,i}u^*_{l,j}\\
  &=\E[x_1bx_1^*]\otimes  \sum\limits_{k=1}^n  u_{k,i}u^*_{k,j}+\E[x_1bx_2^*]\otimes (1_\A-\sum\limits_{k=1}^n  u_{k,i}u^*_{k,j}),\\  
\end{aligned}
$$
which implies that
$$\left(\E[x_1bx_1^*]-\E[x_1bx_2^*]\right)\otimes  \sum\limits_{k=1}^n  u_{k,i}u^*_{k,j}=0$$
for all $b\in \B$. It follows that  $\sum\limits_{k=1}^n  u_{k,i}u^*_{k,j}=0$.  Similarly, $\sum\limits_{k=1}^n  u^*_{k,i}u_{k,j}=0.$

Case (2):   If $\E[x_i]=0$ for all $i$ and $\E[x_ibx_j]=0$ for all $i\neq j, b\in \B$, then
$$
\begin{aligned}
  \E[x_ibx_i^*]\otimes 1_\A&=\sum\limits_{k,l=1}^n \E[x_kbx_l^*]\otimes u_{k,i}u^*_{l,i}\\
  &=\sum\limits_{k=1}^n \E[x_kbx_k^*]\otimes u_{k,i}u^*_{k,i}+\sum\limits_{k,l=1,k\neq l}^n \E[x_kbx_l^*]\otimes u_{k,i}u^*_{l,i}\\
  &=\sum\limits_{k=1}^n \E[x_1bx_1^*]\otimes u_{k,i}u^*_{k,i}
\end{aligned}
$$ 
Thus $\sum\limits_{k=1}^n  u_{k,i}u^*_{k,i}=1_\A$ for all $i$. Similarly, we have $\sum\limits_{k=1}^n  u^*_{k,i}u_{k,i}=1_\A.$ 

On the other hand, for $i\neq j$ and $b\in \B$, we have 
$$
\begin{aligned}
  0=\E[x_ibx_j^*]\otimes 1_\A&=\sum\limits_{k,l=1}^n \E[x_kbx_l^*]\otimes u_{k,i}u^*_{l,j}\\
  &=\sum\limits_{k=1}^n \E[x_kbx_k^*]\otimes u_{k,i}u^*_{k,j}+\sum\limits_{k,l=1,k\neq l}^n \E[x_kbx_l^*]\otimes u_{k,i}u^*_{l,j}\\
  &=\E[x_1bx_1^*]\otimes  \sum\limits_{k=1}^n  u_{k,i}u^*_{k,j}\\  
\end{aligned}
$$
It follows that  $\sum\limits_{k=1}^n  u_{k,i}u^*_{k,j}=0$.  Similarly, $\sum\limits_{k=1}^n  u^*_{k,i}u_{k,j}=0.$

By Lemma \ref{Unitary in M(A)},  $u=(u_{i,j})$ is unitary in $M_n(\A)$.
\end{proof}

It is well known that $q$-Gaussian Brownian motions \cite{BS}, mixed $q$-Gaussian Brownian motions \cite{Sp1993}, and even random variables in a more general deformed commutation relations \cite{BS1994} satisfy the second condition in Proposition \ref{unitary condition}. As a result, the maximal distributional symmetries for these examples arise from quantum unitary groups. 

Another interesting example is that of identically distributed $\epsilon$-free independent \cite{Woj, Sp-Wy} random variables, which satisfy the conditions in Proposition \ref{unitary condition}. Therefore, their distributional symmetries also arise from quantum unitary groups. In fact, the maximal distributional symmetry quantum group for these examples at least contains Speicher and Weber's quantum groups with partial commutation relations \cite{Sp-We}.

On the other hand, every compact quantum group $\A$ admits a Haar state $h$ on $\A$ such that 
 for any sate $\rho$ on $\A$ one has $(\rho \otimes h) \Delta=(h \otimes \rho) \Delta=\rho(I) h$ as established in \cite{Van, Wo2}. This property extends to the corepresentation  $u=(u_{i,j})$ of $\A$.  Fix $j$, consider the algebra $\M_j$ generated by $\{u_{i,j}|i=1,\cdots,k\}$, with Haar state of $\A$ restricted to $\M$. Then, $\{u_{i,j}|i=1,\cdots,k\}$ is $\A$-invariant. Consequently, every corepresentation is contained within a quantum semigroup arising from the maximal distributional symmetry associated with a sequence of random variables. In summary, it seems that we can establish a probabilistic framework for classifying compact quantum groups.

\section{Universal conditions for quantum groups}


In this section, we will study the universal conditions satisfied by generator of a $C^*$-algebra so that make the $C^*$-algebra a compact quantum group.
Let $\A$ be a $C^*$-algebra generated by $n^2$ elements $\{u_{i,j}|i,j=1,\cdots,n\}$ such that $u=(u_{i,j})$ is unitary.  
According to Definition \ref{quantum groups}, $\A$ is a quantum group if the map $u_{i,j}\rightarrow \sum\limits_{k=1}^n u_{i,k}\otimes u_{k,j}$ extends to a homomorphism from $\A$ to $\A\otimes \A$.

\begin{lemma}\label{sum condition} Let $\A$ be a $C^*$-algebra,  $\{u_{i,j}|i,j=1,\cdots,n\}\in \A$, $(d_1,\cdots,d_m)\in \{1,*\}^m$ and $v_{i,j}=\sum\limits_{k=1}^n u_{i,k}\otimes u_{k,j}\in\A\otimes_{\min}\A$. If
$$\sum\limits_{k=1}^n u^{d_1}_{k,i_1} u^{d_2}_{k, i_2} \cdots u^{d_m}_{k, i_m}  =\delta_{i_1, i_2} \cdots \delta_{i_{m-1}, i_m}1_\A$$
for all $i_1,\cdots,i_m\in [n]$,
then $$\sum\limits_{k=1}^n v^{d_1}_{k,i_1} v^{d_2}_{k, i_2} \cdots v^{d_m}_{k, i_m}  =\delta_{i_1, i_2} \cdots \delta_{i_{m-1}, i_m} 1_\A\otimes 1_\A$$
for all $i_1,\cdots,i_m\in [n]$.
\end{lemma}
\begin{proof} By straightforward computation, we have 
$$
\begin{aligned}
&\sum\limits_{k=1}^n v^{d_1}_{k,i_1} v^{d_2}_{k, i_2} \cdots v^{d_m}_{k, i_m} \\
=&\sum\limits_{k=1}^n 
\left(\sum\limits_{j_1=1}^nu^{d_1}_{k, j_1}\otimes u^{d_1}_{j_1,i_1}\right)
\cdots
\left(\sum\limits_{j_m=1}^nu^{d_m}_{k, j_m}\otimes u^{d_m}_{j_m,i_m}\right)\\
=&
\sum\limits_{j_1=1}^n \cdots\sum\limits_{j_m=1}^n \sum\limits_{k=1}^n 
\left(u_{k, j_1}^{d_1}\cdots u^{d_m}_{k, j_m}\otimes u^{d_1}_{j_1,i_1}\cdots u^{d_m}_{j_m,i_m}\right)
 \\
 =&
\sum\limits_{j_1=1}^n \cdots \sum\limits_{j_m=1}^n 
\left(\delta_{j_1, j_2} \cdots \delta_{j_{m-1}, j_m}\otimes u^{d_1}_{j_1,i_1}\cdots u^{d_m}_{j_m,i_m}\right)\\
=&\sum\limits_{j=1}^n  \left(1_\A \otimes u^{d_1}_{j,i_1}\cdots u^{d_m}_{j,i_m}\right)\\
=&\delta_{i_1, i_2} \cdots \delta_{i_{m-1}, i_m}(1_\A\otimes 1_\A)
\end{aligned}
$$
\end{proof}

Lemma \ref{sum condition} provides lots of universal conditions for making $\A$ a quantum group. However, these conditions are not independent and we will show that each of these conditions can be reduced to one of the following cases.

\begin{proposition}
Let $\A$ be a $C^*$-algebra generated by $n^2$ elements $\{u_{i,j}|i,j=1,\cdots,n\}$ such that $u=(u_{i,j})$ is unitary in $M_n(\A)$.  If the elements $\{u_{i,j}|i,j=1,\cdots,n\}$ satisfy one of the following conditions, then
the set $\{\sum\limits_{k=1}^n u_{i,k}\otimes u_{k,j}|i,j=1,\cdots,n\}$ also satisfy the corresponding condition.
\begin{enumerate}
\item $\sum\limits_{k=1}^n u_{k, i_1} u_{k, i_2} =\delta_{i_1, i_2}$,  for all $i_1,,i_2\in [n]$
\item  $\sum\limits_{k=1}^n u_{k, i_1} u^*_{k, i_2} u_{k, i_3} u_{k, i_4}^*=\delta_{i_1, i_2} \cdots \delta_{i_3, i_4}$,  for all $i_1,\cdots,i_4\in [n]$
\item  $\sum\limits_{k=1}^n u_{k, i_1} u_{k, i_2} u_{k, i_3}^* u_{k, i_4}^*=\delta_{i_1, i_2} \cdots \delta_{i_3, i_4}$,  for all $i_1,\cdots,i_4\in [n]$
\item For fixed $m\in \mathbb{Z}^+$:$\sum\limits_{k=1}^n u_{k,i_1} u_{k, i_2} \cdots u_{k, i_m}  =\delta_{i_1, i_2} \cdots \delta_{i_{k-1}, i_k}$, for all $i_1,\cdots,i_m\in [n]$
\end{enumerate} 
\end{proposition}
\begin{proof}  
\begin{enumerate}
\item Apply Lemma \ref{sum condition}, by letting  $m=2$ and $\dd=(1,1)$.
\item Apply Lemma \ref{sum condition}, by letting  $m=4$ and $\dd=(1,*,1,*)$. 
\item Apply Lemma \ref{sum condition}, by letting  $m=4$ and $\dd=(1,1,*,*)$. 
\item  Apply Lemma \ref{sum condition}, by letting $d_i=1$ for all $i$. 
\end{enumerate}
\end{proof}

\begin{lemma}
Assume that $\A$ is a $C^*$-algebra generated by $n^2$ elements $\{u_{i,j}|i,j=1,\cdots,n\}$ such that $u=(u_{i,j})$ is unitary in $M_n(\A)$ and $\sum\limits_{k=1}^n u_{k, i_1} u_{k, i_2} =\delta_{i_1, i_2}$,  for all $i_1,,i_2\in [n]$. Then, $u_{i,j}=u_{i,j}^*$ for all $i,j=1,\cdots,n$.
\end{lemma}
\begin{proof}
 Notice that $\sum\limits_{k=1}^n u_{k, i_1} u_{k, i_2} =\delta_{i_1, i_2}$ is equivalent to  $u^Tu=1_n$, where $1_n$ is the identity of $M_n(\A)$. Therefore, $$(u_{j,i})=u^T=u^{-1}=u^*=(u_{j,i}^*).$$
 By comparing the entries of the matrices, we get $u_{i,j}=u_{i,j}^*$ for all $i,j=1,\cdots,n$.
 
\end{proof}

Based on the above relations between $u_{i,j}$ and $\sum\limits_{k=1}^n u_{i,j}$,  we have the following definitions for quantum groups.

\begin{definition}\label{quantum groups via vanishing conditions}
 Let $C(\qs,n)$, $C(\qo,n)$, $C(\qb_s,n)$, $C(\qh_s,n)$, $C(\qb,n) $, $C(\qh_m,n)$ with $m\geq 3$, $C(\qh_0,n)$ and $C(\rqh,n)$ be the universal $C^*$-algebras generated by $\{u_{i,j}|i,j=1,\cdots,n\})$, such that they are quantum subgroups of $C(\qu,n)$ and satisfy  the following conditions respectively:

Selfadjoint case: In the follow, we assume that $\sum\limits_{k=1}^n u_{k, i_1} u_{k, i_2} =\delta_{i_1, i_2}$,  for all $i_1,,i_2\in [n]$. Then, $u_{i,j}=u_{i,j}^*$ for all $i,j=1,\cdots,n$. 
 \begin{enumerate}
  \item $C(\qs,n)$:   $u_{i,j}$'s are orthogonal projections, namely, $u_{i,j}=u^*_{j,i}=u^2_{i,j}$ for all $i,j=1,\cdots,n$.
\item $C(\qo,n)$:    no other conditions to satisfy.
\item $C(\qb_s,n)$:   $\sum\limits_{k=1}^n u_{i,k}=\sum\limits_{k=1}^n u_{k,j}=1 $  for all $i,j=1,\cdots,n$.
\item $C(\qh_s,n)$:  $u^2_{i,j}$'s are orthogonal projections, for all $i,j=1,\cdots,n$.
 \end{enumerate}

Non-selfadjoint case: 
 \begin{enumerate}
\item $C(\qb,n) $:   $\sum\limits_{k=1}^n u_{i,k}=\sum\limits_{k=1}^n u_{k,j}=1 $  for all $i,j=1,\cdots,n$.
\item $C(\qh_m,n)$ with $m\geq 3$:$
\sum\limits_{k=1}^n u_{k,i_1} u_{k, i_2} \cdots u_{k, i_m}  =\delta_{i_1, i_2} \cdots \delta_{i_{m-1}, i_m}$, for all $i_1,\cdots,i_m\in [n]$.
\item $C(\qh_0,n)$:  $\sum\limits_{k=1}^n u_{k, i_1} u_{k, i_2} u_{k, i_3}^* u_{k, i_4}^*=\delta_{i_1, i_2} \cdots \delta_{i_3, i_4}$,  for all $i_1,\cdots,i_4\in [n]$.
\item $C(\rqh,n)$: $\sum\limits_{k=1}^n u_{k, i_1} u^*_{k, i_2} u_{k, i_3} u_{k, i_4}^*=\delta_{i_1, i_2} \cdots \delta_{i_3, i_4}$,  for all $i_1,\cdots,i_4\in [n]$.
\end{enumerate}  
\end{definition}

\begin{remark}
For $C(\qh_m,n)$ with $m=2$, we obtain $C(\qo,n)$; and with $m=1$, we obtain $C(\qb,n)$. The sub-index $s$ n the preceding definition indicates that the generators of the corresponding algebra are self-adjoint.
\end{remark}
According to the universal relations satisfied by the universal quantum groups mentioned above, we can establish the following relationship between them.

\begin{proposition}\label{quantum subgroup relations}
For Fixed $n\geq 1$, we have 
\begin{enumerate}
\item $ C(\qs,n)\subset C(\qb_s,n)\subset C(\qb,n)\subset C(\qu,n)$.
\item $ C(\qs,n)\subset C(\qh_s,n)\subset C(\qo,n)\subset C(\qu,n)$.
\item $ C(\qb_s,n)\subset C(\qo,n)$.
\item For $m\geq 3$: $ C(\qs,n)\subset C(\qh_m,n)\subset C(\qh_0,n)\subset C(\rqh,n)\subset C(\qu,n)$
\end{enumerate}
\end{proposition}

\begin{proof}
According to the definitions, 1)-3) are obvious. 
 In addition, we have $ C(\qs,n)\subset C(\qh_m,n)$ and $C(\rqh,n)\subset C(\qu,n)$. It remains to show that  $C(\qh_m,n)\subset C(\qh_0,n)\subset C(\rqh,n)$ for all $m\geq 3$.  For convenience, we denote by $\delta(i_1,\cdots,i_k)=\delta_{i_1, i_2} \cdots \delta_{i_{k-1}, i_k}$ for all $k\geq 2$.
 
Assume that 
$
\sum\limits_{k=1}^n u_{k,i_1} u_{k, i_2} \cdots u_{k, i_m}  =\delta(i_1, \cdots, i_m)$  for all $i_1,\cdots,i_m\in [n]$, then 
$$
\sum\limits_{k=1}^n u^*_{k,i_m} \cdots u^*_{k,i_1}  =\delta(i_1, \cdots, i_m)$$ 
for all $i_1,\cdots,i_m\in [n]$.
It follows that 
$$ \left(\sum\limits_{k=1}^n u_{k,j_1}  \cdots u_{k, j_m}  \right) \left(\sum\limits_{k'=1}^n u^*_{k',i_m} \cdots u^*_{k',i_1}\right)
=\delta(i_1, \cdots, i_m)\delta(j_1, \cdots, j_m),$$

Let $j_m=i_m$. Then we have 
\begin{equation}\label{eq1}
\left(\sum\limits_{k=1}^n u_{k,i_1}u_{k,j_2}  \cdots u_{k, j_m}  \right) \left(\sum\limits_{k'=1}^n u^*_{k',i_m} \cdots u^*_{k',i_1}\right)=\delta(i_1, \cdots, i_m, j_1, j_2,\cdots, j_{m-1}),
\end{equation} 

Since $\sum\limits_{i_m=1}^n u^*_{k',i_m} u_{k,i_m}=\delta_{k',k}$,  and 
$\sum\limits_{i_m=1}^n \delta(i_m,j_{m-1}, i_{m-1})=\delta_{i_{m-1},j_{m-1}}$, by taking the sum of Equation (\ref{eq1}) over $i_{m}$, we have 
\begin{equation}\label{eq2}
\sum\limits_{k=1}^n u_{k,i_1} u_{k,j_2}  \cdots u_{k, j_{m-1}} u^*_{k,i_{m-1}} \cdots u^*_{k,i_1}=\delta(i_1,\cdots, i_{m-1},j_1,  \cdots,j_{m-1}),
\end{equation} 

Similarly, by letting $i_3=j_3,\cdots, i_{m-1}=j_{m-1}$ and take the sum of the Equation(\ref{eq2}) over $i_2,\cdots, i_{m-2}$ from $1$ to $n$, we have  
$$\sum\limits_{k=1}^n u_{k, i_1} u_{k, i_2} u_{k, j_2}^* u_{k, j_1}^*=\delta(i_1, i_2, j_1, j_2).$$
It follows that $C(\qh_m,n)\subset C(\qh_0,n)$.  Now, it remains to show that $C(\qh_0,n)\subset C(\rqh,n)$.

Assume that  $\sum\limits_{k=1}^n u_{k, i_1} u_{k, i_2} u_{k, i_3}^* u_{k, i_4}^*=\delta(i_1,\cdots ,i_4)$,  for all $i_1,\cdots,i_4\in [n]$, then we have 
\begin{equation}\label{eq3}
\left(\sum\limits_{k=1}^n u_{k, i_1} u_{k, i_2} u_{k, i_3}^* u_{k, i_4}^* \right) \left(\sum\limits_{k'=1}^n u_{k', j_1} u_{k', j_2} u_{k', j_3}^* u_{k', j_4}^* \right)=\delta(i_1,\cdots,i_4)\delta(j_1,\cdots,j_4).
\end{equation}
Again, by letting $i_4=j_1$ and take the sum of Equation(\ref{eq3}) over $j_1$, we get
\begin{equation}\label{eq4}
\sum\limits_{k=1}^n u_{k, i_1} u_{k, i_2} u_{k, i_3}^*  u_{k, j_2} u_{k, j_3}^* u_{k, j_4}^* =\delta(i_1,\cdots,i_3, j_2,\cdots,j_4).
\end{equation}
Notice that 
$$
\begin{aligned}
&\sum\limits_{i_1=1}^n\left[\left(\sum_{k_1=1}^n u_{k_1,i_1}^*\right) \left(\sum\limits_{k=1}^n u_{k, i_1} u_{k, i_2} u_{k, i_3}^*  u_{k, j_2} u_{k, j_3}^* u_{k, j_4}^*\right) \right]\\
=&\sum_{k_1=1}^n\sum\limits_{k=1}^n \left(\sum\limits_{i_1=1}^n u_{k_1,i_1}^*u_{k, i_1} \right)u_{k, i_2} u_{k, i_3}^*  u_{k, j_2} u_{k, j_3}^* u_{k, j_4}^*\\
=&\sum_{k=1}^n\sum\limits_{k_1=1}^n \delta_{k,k_1} u_{k, i_2} u_{k, i_3}^*  u_{k, j_2} u_{k, j_3}^* u_{k, j_4}^*\\
=&\sum_{k=1}^n u_{k, i_2} u_{k, i_3}^*  u_{k, j_2} u_{k, j_3}^* u_{k, j_4}^*\\
\end{aligned}
$$
and 
$$
\sum\limits_{i_1=1}^n\left[\left(\sum_{k_1=1}^n u_{k_1,i_1}^*\right) \left(\delta(i_1,\cdots,i_3, j_2,\cdots,j_4)\right) \right]=\left(\sum_{k_1=1}^n u_{k_1,j_4}^*\right)\delta(i_2,i_3, j_2,\cdots,j_4).
$$

Thus, we have 
$$ \sum\limits_{k=1}^n u_{k, i_2} u_{k, i_3}^*  u_{k, j_2} u_{k, j_3}^* u_{k, j_4}^*=\left(\sum_{k_1=1}^n u_{k_1,j_4}^*\right)\delta(i_2,i_3, j_2,\cdots,j_4).$$

Similarly, we have 
$$\sum\limits_{j_4=1}^n\left[\left(\sum\limits_{k=1}^n u_{k, i_2} u_{k, i_3}^*  u_{k, j_2} u_{k, j_3}^* u_{k, j_4}^*\right) \left(\sum_{k_2=1}^n u_{k_2,j_4}\right) \right]=\sum_{k=1}^n u_{k, i_2} u_{k, i_3}^*  u_{k, j_2} u_{k, j_3}^* $$
and 
$$
\begin{aligned}
\sum\limits_{j_4=1}^n\left[ \left(\sum_{k_1=1}^n u_{k_1,j_4}^*\right)\delta(i_2,i_3, j_2,\cdots,j_4)\left(\sum_{k_2=1}^n u_{k_2,j_4}\right) \right]=\sum_{k_1=1}^n\sum_{k_2=1}^n u_{k_1,j_3}^* u_{k_2,j_3} \delta(i_2,i_3,j_2,j_3)
\end{aligned}
$$
Thus, we have 
\begin{equation}\label{eq5}
\sum_{k=1}^n u_{k, i_2} u_{k, i_3}^*  u_{k, j_2} u_{k, j_3}^*=\sum_{k_1=1}^n\sum_{k_2=1}^n u_{k_1,j_3}^* u_{k_2,j_3} \delta(i_2,i_3,j_2,j_3)
\end{equation}
If $\delta(i_2,i_3,j_2,j_3)=0$, then we have $\sum\limits_{k=1}^n u_{k, i_2} u_{k, i_3}^*  u_{k, j_2} u_{k, j_3}^* =0=\delta(i_2,i_3,j_2,j_3).$\\
If $\delta(i_2,i_3,j_2,j_3)=1$, then $i_2=i_3=j_2=j_3$. Since $\sum\limits_{k=1}^n u_{k, j_3} u_{k, j_3}^*=1_{\A}$ and $u_{k, j_3} u_{k, j_3}^*\geq 0$ for all $k$,  $u_{k, j_3} u_{k, j_3}^*\leq 1_{\A}$. It follows that 
\begin{equation}\label{eq6}
\sum\limits_{k=1}^n u_{k, j_3} u_{k, j_3}^*u_{k, j_3} u_{k, j_3}^*\leq \sum\limits_{k=1}^n u_{k, j_3} u_{k, j_3}^*=1_{\A}.
\end{equation}
Hence,
\begin{equation}\label{eq7}
     \sum\limits_{j_3=1}^n\sum\limits_{k=1}^n u_{k, j_3} u_{k, j_3}^*u_{k, j_3} u_{k, j_3}^*\leq n1_{\A}.
     \end{equation}

On the other hand, put $i_2=i_3=j_2=j_3$ in Equation (\ref{eq6}), we have 
$$
\sum\limits_{j_3=1}^n\sum\limits_{k=1}^n u_{k, j_3} u_{k, j_3}^*u_{k, j_3} u_{k, j_3}^*
=\sum\limits_{j_3=1}^n\sum_{k_1=1}^n\sum_{k_2=1}^n u_{k_1,j_3}^* u_{k_2,j_3} 
=\sum_{k_1=1}^n\sum_{k_2=1}^n \delta_{k_1,k_2}=n1_{\A}
$$

Therefore,  the equality holds in the inequality (\ref{eq6}).  We always have
$$\sum\limits_{k=1}^n u_{k, i_1} u^*_{k, i_2} u_{k, i_3} u_{k, i_4}^*=\delta_{i_1, i_2} \cdots \delta_{i_3, i_4},$$  for all $i_1,\cdots,i_4\in [n]$.  It follows that $C(\qh_0,n)\subset C(\rqh,n)$. 
\end{proof}

In summary, for each $n$, we have the following diagram.

\begin{center}

\begin{tikzpicture}[xscale=1.5,yscale=1]

\path 
node   (m11) at (0,2) {$C(\qs,n)$}         
node   (m21) at (1.5,3) {$C(\qb_s,n)$}
node   (m22) at (1.5,2) {$C(\qh_s,n)$}
node   (m23) at (1.5,1) {$C(\qh_m,n)$}
node   (m33) at (3,1) {$C(\qh_0,n)$}
node   (m32) at (3,2) {$C(\qo,n)$}
node   (m31) at (3,3) {$C(\qb,n)$}     
node   (m42) at (6,2) {$C(\qu,n)$}
node   (m43) at (4.5,1) {$C(\rqh,n)$}                               
 ;
{[->]       
\draw   (m11) --  (m21); 
\draw   (m11) --  (m22);
\draw   (m11) --  (m23); 
\draw   (m21) --  (m31);
\draw   (m21) --  (m32);
\draw   (m23) --  (m33);
\draw   (m32) --  (m42);
\draw   (m33) --  (m43);
\draw   (m31) --  (m42);
\draw   (m43) --  (m42);
\draw   (m22) --  (m32);
      }
   
\end{tikzpicture}
\end{center}

In the above diagram, $\A\rightarrow \B$ means $\A$ is a quantum subgroup of $\B$.

Now, we provide examples to show that the above relations are proper for $n\geq 3$. 

\begin{example}\label{non-isomorphic example}
To show that $\A$ is a proper quantum subgroup of $\mathcal{B}$, it suffices  to provide a matrix representation of $\mathcal{B}$ that  is not of $\A$.
\begin{enumerate}
\item For $ C(\qs,n)\subsetneq C(\qb_s,n)\subsetneq C(\qb,n)\subsetneq C(\qu,n)$: Let $n=3$, and consider the following matrix representations for $C(\qb_s,n), C(\qb,n), C(\qu,n)$:
$$\left(\begin{array}{ccc}
\frac{2}{3} & \frac{2}{3} & \frac{-1}{3} \\
\frac{2}{3} & -\frac{1}{3} & \frac{2}{3} \\
-\frac{1}{3} & \frac{2}{3} & \frac{2}{3}
\end{array}\right) ,
\left(\begin{array}{ccc}
\frac{1}{2}+i & \frac{1}{2}-i &0\\
\frac{1}{2}-i & \frac{1}{2}+i &0\\
0&0&1
\end{array}\right),
\left(\begin{array}{ccc}
i &0 &0\\
0 & 1 &0\\
0&0&1
\end{array}\right)
$$
\item For $ C(\qs,n)\subsetneq C(\qh_s,n)\subsetneq C(\qo,n)\subsetneq C(\qu,n)$. Let $n=2$, and consider the following matrix representations for $ C(\qh_s,n), C(\qo,n), C(\qu,n)$:
$$
\left(\begin{array}{cc}
-1&0\\
0 & 1
\end{array}\right),
\left(\begin{array}{cc}
0.6 &0.8\\
0.8 &-0.6
\end{array}\right),
\left(\begin{array}{cc}
i &0\\
0 & 1
\end{array}\right).$$
\item For $ C(\qb_s,n)\subsetneq C(\qo,n)$. Let $n=2$ and consider the following matrix representations for $ C(\qo,n)$:
$$ \left(\begin{array}{cc}
0.6 &0.8\\
0.8 &-0.6
\end{array}\right)$$
\item For $m\geq 3$, consider the following matrix representations:\\
Let $u_{i,j}=\delta_{i,j} e^{\theta\pi i}$ such that  $\theta=\frac{2}{m}$, Then $(u_{i,j})$ is a representation of $C(\qh_m,n)$ but not of $C(\qs,n)$.\\
Let $u_{i,j}=\delta_{i,j} e^{\theta\pi i}$ such that $\theta$ be a irrational real number. Then, $\sum\limits_{\alpha=1}^n u^{d_1}_{\alpha,1}\cdots u^{d_k}_{\alpha,1}\neq 1$
if the numbers of $*$ and $1$ of  $\dd$ are not equal.  In this case, $(u_{i,j})$ is a representation of $C(\qh_0,n)$ but not of $C(\qh_m,n)$.\\
Consider the following representation of $C(\rqh,n)$: Let $u_{i,j}\in M_2(\C)$ such that  
$$ u_{1,1}=\left(\begin{array}{ll}
0 & 1 \\
0 & 0
\end{array}\right) \,
 u_{1,2}=\left(\begin{array}{ll}
0 & 0 \\
1 & 0
\end{array}\right) \,
 u_{2,1}=\left(\begin{array}{ll}
0 & 0 \\
1 & 0
\end{array}\right) \,
 u_{2,2}=\left(\begin{array}{ll}
0 & 1 \\
0 & 0
\end{array}\right) \,$$
 and $u_{i,j}=\delta_{i,j} 1_{M_2(\C)}$ for the other $i,j$. In this case, we have $u_{\alpha,1}^2=(u^*_{\alpha,1})^2=0$ for all $\alpha$, it follows that $\sum\limits_{\alpha=1}^n u^{d_1}_{\alpha,1}\cdots u^{d_k}_{\alpha,1}=1$
only if  $\dd$ is alternating and $|\dd|$ is even.  In this case, $(u_{i,j})$ is a representation of $C(\qh_0,n)$ but not of $C(\rqh,n)$.\\
Notice that  $\left(\begin{array}{cc}
0.6 &0.8\\
0.8 &-0.6
\end{array}\right),$ is a matrix representation of $C(\qu,n)$ but not of $C(\rqh,n)$.\\
Therefore, we have $ C(\qs,n)\subsetneq C(\qh_m,n)\subsetneq C(\qh_0,n)\subsetneq C(\rqh,n)\subsetneq C(\qu,n).$
\end{enumerate}
\end{example}

The characterization of random variables for de Finetti-type theorems relies on vanishing conditions on cumulants, which are described by partitions. Therefore, we need the following identity characterization for the generators of quantum groups over blocks, allowing us to extend these identities to non-crossing partitions and general partitions.

\begin{proposition}\label{quantumgroup-Block identity}

 Let $\A(n)$ be a quantum group generated by $\{u_{i,j}|i,j=1,\cdots,n\})$. Then, we have the following identities.\\
 \begin{itemize}
 \item Selfadjoint cases: 
 \begin{enumerate}
  \item If $\A(n)=C(\qs,n)$, then $\sum\limits_{\alpha=1}^n u^k_{\alpha,i}=1$ for all $k\in \mathbb{Z}^+$.
\item If $\A(n)=C(\qo,n)$, then $\sum\limits_{\alpha=1}^n u^k_{\alpha,i}=1$ if and only if $k=2$.
\item If $\A(n)=C(\qb_s,n)$, then $\sum\limits_{\alpha=1}^n u^k_{\alpha,i}=1$ if and only if  $k=1,2$.
\item If $\A(n)=C(\qh_s,n)$, then $\sum\limits_{\alpha=1}^n u^k_{\alpha,i}=1$ if and only if  $k$ is even.
 \end{enumerate}

\item Non-selfadjoint cases: Let $k\in \mathbb{Z}^+$ and $\dd=(d_1,\cdots,d_k)$ be a sequence of k elements from $\{1,*\}$.
 \begin{enumerate}
\item If $\A(n)=C(\qb,n) $: For $1\leq i\leq n$,  $\sum\limits_{k=1}^n u^{d_1}_{\alpha,i}\cdots u^{d_k}_{\alpha,i}=1$ if and only if  $k=1$ or $k=2$ and $\dd=\{1,*\}\,\text{or}\, \{*,1\}$.

\item If $\A(n)=C(\qh_m,n)$ with $m\geq 3$:  For $1\leq i\leq n$, $\sum\limits_{\alpha=1}^n u^{d_1}_{\alpha,i}\cdots u^{d_k}_{\alpha,i}=1$  if and only if 
the difference of the numbers of $*$ and $1$ of  $\dd$ can be divided by $m$.

\item If $\A(n)=C(\qh_0,n) $: For $1\leq i\leq n$,  $\sum\limits_{\alpha=1}^n u^{d_1}_{\alpha,i}\cdots u^{d_k}_{\alpha,i}=1$ if and only if  the numbers of $*$ and $1$ of $\dd$ are equal.

\item If $\A(n)=C(\rqh,n)$:  For $1\leq i\leq n$, $\sum\limits_{\alpha=1}^n u^{d_1}_{\alpha,i}\cdots u^{d_k}_{\alpha,i}=1$ if and only if  $k\in 2\mathbb{Z}^+$ and $\dd$ is alternating.

\item If $\A(n)=C(\qu,n) $:  For $1\leq i\leq n$, $\sum\limits_{\alpha=1}^n u^{d_1}_{\alpha,i}\cdots u^{d_k}_{\alpha,i}=1$  if and only if  $k=2$ and $\dd=\{1,*\}\,\text{or}\, \{*,1\}$.

\end{enumerate}  
 \end{itemize}
 
\end{proposition}

\begin{proof}
For selfadjoint cases, see\cite{Liu3}.  In the following, we will focus on the non-selfadjoint cases: 

1)  By the definition of  $C(\qb, n)$, we have $\sum\limits_{\alpha=1}^n u^{d_1}_{\alpha,1}\cdots u^{d_k}_{\alpha,1}=1$ if $k=1$ or $k=2$ and $\dd=\{1,*\}\,\text{or}\, \{*,1\}$.  \\
Notice that $C(\qb_s,n)\subset C(\qb,n)$, the equation hold only if $k=1,2$.  When $k=2$, consider the representation in Example \ref{non-isomorphic example}. Let $$u_{1,1}=\frac{1}{2}+\frac{i}{2}, u_{1,2}=\frac{1}{2}-\frac{i}{2},u_{2,1}=\frac{1}{2}-\frac{i}{2},u_{2,2}=\frac{1}{2}+\frac{i}{2}.$$ 
For the other $u_{i,j}$, let $u_{i,j}=\delta_{i,j}$. Then, we have that  $\sum\limits_{\alpha=1}^n u^2_{\alpha,1}\neq 1$ and   $\sum\limits_{\alpha=1}^n (u^*)^2_{\alpha,1}\neq 1$.

2)  According to the definition of $(\qh_m,n)$, we have that 
$$\sum\limits_{\alpha=1}^n u_{\alpha,i}u^{m-1}_{\alpha,j}=\delta_{i,j}.$$
Since $(u_{i,j})$ is unitary in $M_n(\A)$, we have that $u^*_{i,j}=u^{m-1}_{i,j}$. It follows that  $u^{d_1}_{\alpha,1}\cdots u^{d_k}_{\alpha,1}=u^{lm}_{\alpha,1}$ 
for some $l\in \mathbb{Z}^+$ if  the difference of the numbers of $*$ and $1$ of  $\dd$ can be divided by $m$.

On the hand, we have 
$ u_{i,j}=(u^*_{i,j})^{m-1}=u^{(m-1)^2}_{i,j}.$ It follows that $ u^{lm}_{\alpha,i}=u^{l'm}_{\alpha,i}$ for some $l'\leq m-2$.

In Equation (\ref{eq2}), let $i_1=i_2=\cdots-i_{l'}=1$ and take the sum of the equation over $j_{l'+1},\cdots,j_{m-1}$ from $1$ to $n$.
We get 
$$1=\sum\limits_{\alpha=1}^n u^{l'}_{\alpha,i}(u^*_{\alpha,i})^{l'}=\sum\limits_{\alpha=1}^n u^{l'm}_{\alpha,i}.$$

Let $u_{1,1}=e^{\frac{2\pi i}{m}}$, and $u_{i,j}=\delta_{i,j}$ for the other $u_{i,j}$. Then, $\sum\limits_{\alpha=1}^n u^{d_1}_{\alpha,i}\cdots u^{d_k}_{\alpha,1}\neq 1$
if 
the difference of the numbers of $*$ and $1$ of  $\dd$ can not be divided by $m$.

3) Applying the method from last case, by letting $v=(u^2_{i,j}u^*_{i,j})_{i,j=1,\cdots,n}\in M_n(\A)$, we get $vu=uv=1_n$. It follows that
$u_{i,j}=u^2_{i,j}u^*_{i,j}$ and $u_{i,j}^*=u_{i,j}(u^*_{i,j})^2$ and 
$$\sum\limits_{\alpha=1}^n u^*_{\alpha,j}u^2_{\alpha,i}u^*_{\alpha,i}=\delta_{i,j}.$$

Since $\sum\limits_{\alpha=1}^n u_{\alpha, i_1} u_{\alpha, i_2} u_{\alpha, i_3}^* u_{\alpha, i_4}^*=\delta(i_1, i_2,i_3, i_4)$
and $\sum\limits_{\alpha'=1}^n u_{\alpha', i_5}^* u_{\alpha', i_4}=\delta_{i_5, i_4}$, we have 
$$\sum\limits_{\alpha=1}^n \sum\limits_{\alpha'=1}^n u_{\alpha', i_5}^* u_{\alpha, i_1} u_{\alpha, i_2} u_{\alpha, i_3}^* u_{\alpha, i_4}^* u_{\alpha', i_4}=\delta(i_1, i_2,i_3, i_4,i_5). $$
It follows that $$
\begin{aligned}
\delta(i_1, i_2,i_3,i_5)
=&\sum\limits_{i_4=1}^n \delta(i_1, i_2,i_3, i_4,i_5)\\
=&\sum\limits_{i_4=1}^n \sum\limits_{\alpha=1}^n \sum\limits_{\alpha'=1}^n u_{\alpha', i_5}^* u_{\alpha, i_1} u_{\alpha, i_2} u_{\alpha, i_3}^* u_{\alpha, i_4}^* u_{\alpha', i_4}\\
=&\sum\limits_{\alpha=1}^n \sum\limits_{\alpha'=1}^n u_{\alpha', i_5}^* u_{\alpha, i_1} u_{\alpha, i_2} u_{\alpha, i_3}^* \delta_{\alpha,\alpha'}\\
=&\sum\limits_{\alpha=1}^n  u_{\alpha, i_5}^* u_{\alpha, i_1} u_{\alpha, i_2} u_{\alpha, i_3}^*.
\end{aligned}
$$
Similarly, we have 
$$ \delta(i_1, i_2,i_5,i_6)=\sum\limits_{i_3=1}^n \sum\limits_{\alpha=1}^n \sum\limits_{\alpha'=1}^n u_{\alpha', i_6}^*u_{\alpha, i_5}^* u_{\alpha, i_1} u_{\alpha, i_2} u_{\alpha, i_3}^* u_{\alpha', i_3}=\sum\limits_{\alpha=1}^n u_{\alpha, i_6}^*u_{\alpha, i_5}^* u_{\alpha, i_1} u_{\alpha, i_2}.$$
Follows the  proof for $C(\qh_0,n)\subseteq C(\rqh,n)$, we have 
$$ \sum\limits_{\alpha=1}^n u_{\alpha, i}^* u_{\alpha, i}u_{\alpha, i}^* u_{\alpha, i}=1$$
and $u_{\alpha, i}^* u_{\alpha, i}$ is an orthogonal projection for all $i, \alpha$.

Let $v'=(u^*_{i,j}u^2_{i,j})_{i,j=1,\cdots,n}\in M_n(\A)$. Then, we have $u^*v'=1_{M_n(\A)}$. It follows that $u_{i,j}=u^*_{i,j}u^2_{i,j}$ and $u_{i,j}^*=(u^*_{i,j})^2u_{i,j}$.  Therefore, we have 
$$ \sum\limits_{\alpha=1}^n u_{\alpha,j}\left(u^*_{\alpha,i}\right)^2u_{\alpha,i}=\delta_{i,j}.$$
Now, we finish the proof to the sufficient part by induction.   The equality holds for $k=2$, since $u$ is unitary. For $k=4$, $\dd$ can be one of  
$$(1,*,1,*), (*,1,*,1) ,(1,*,*,1),  (1,1,*,*) , (*,*,1,1), (*,1,1,*),$$ 
the equation holds all these cases.\\
Assume that the equation holds for $k=2l$ where $l\geq 2$. \\
When $k=2l+2$. 
If  $\dd$ is alternating, then $\sum\limits_{\alpha=1}^n u^{d_1}_{\alpha,i}\cdots u^{d_k}_{\alpha,i}$ is equal to either $\sum\limits_{\alpha=1}^n \left(u_{\alpha,i}u^*_{\alpha,i}\right)^k$ or $\sum\limits_{\alpha=1}^n \left(u^*_{\alpha,i}u_{\alpha,i}\right)^k$, thus equal to either  $\sum\limits_{\alpha=1}^n u_{\alpha,i}u^*_{\alpha,i}$ or $\sum\limits_{\alpha=1}^n u^*_{\alpha,i}u_{\alpha,i}$, which are both equal to $1_{\A}$.\\
If $\dd$ is not alternating, since the the numbers of $*$ and $1$ of $\dd$ are equal,  we have three consecutive items in $\dd$ equal to $(1,1,*),(*,*,1),(*,1,1),(1,*,*)$.  Since  $u_{i,j}^*=(u^*_{i,j})^2u_{i,j}=u_{i,j}(u^*_{i,j})^2$ for all $i,j$. The case reduces to the case for $k=2l$. By induction, the equation holds.\\
The necessary part for the statement follows the representation in Example \ref{non-isomorphic example}.

4)  Since $\sum\limits_{\alpha=1}^n u_{\alpha, i_1} u^*_{\alpha, i_2} u_{\alpha, i_3} u_{\alpha, i_4}^*=\delta_{i_1, i_2,i_3, i_4}$,  for all $i_1,\cdots,i_4\in [n]$, apply the method in the previous case, we have 
$$ \delta_{i_1, i_2,i_3,i_5}=
\sum\limits_{i_4=1}^n \sum\limits_{\alpha=1}^n \sum\limits_{\alpha'=1}^n u_{\alpha', i_5}^*u_{\alpha, i_1} u^*_{\alpha, i_2} u_{\alpha, i_3} u_{\alpha, i_4}^*u_{\alpha', i_4}
=\sum\limits_{\alpha=1}^n u_{\alpha, i_5}^*u_{\alpha, i_1} u^*_{\alpha, i_2} u_{\alpha, i_3}.$$
Thus, $u_{i,j}u_{i,j}^*$ and $u^*_{i,j}u_{i,j}$ are projections for all $i,j$. If  $\dd$ is alternating, follows the previous case, then $\sum\limits_{\alpha=1}^n u^{d_1}_{\alpha,i}\cdots u^{d_k}_{\alpha,i}$ is equal to $1_{\A}$.\\
The necessary part for this  statement also follows the representation in Example \ref{non-isomorphic example}.

5)By the  definition of $C(\qu,n)$, both $C(\qo,n)$ and $C(\rqh,n)$ are quantum subgroups of $C(\qu,n)$,  thus the equation holds only if $k=2$ and $\dd$ is alternating. 

\end{proof}

\section{ Vanishing-cumulants for single random variables}

The classification results presented in this paper are closely linked to the vanishing cumulants condition.
 Therefore, we need to consider the classification of random variables based on cumulants. 
 In the definition below, we introduce a classification based on non-vanishing free cumulants in free probability. 
 For the definitions in classical probability, one merely needs to replace free cumulants with classical cumulants.

\begin{definition}\label{types of random variables} Let $x$ be an element of a $\B$-valued probability space $(\A, \E)$.
\begin{enumerate}

\item If  for each $k\in \N$ and $b_1,\cdots,b_k\in \B$, we have 
$$\kappa^{\pi}(x^{d_1}b_1,\cdots,x^{d_k}b_k)=0$$
unless all blocks of $\pi$ are of even sizes,  then we say that $x$ is a \textit{symmetric} element.

\item If  for each $k\in \N$, $\dd\in \{1,*\}^k$ and $b_1,\cdots,b_k\in \B$, we have 
$$\kappa^{(\pi)}(x^{d_1}b_1,\cdots,x^{d_k}b_k)=0$$
unless $\pi\in  NC_2(k),$  then we say $x$ is an \textit{orthogonal} element. Further, if $x$ is selfadjoint, then $x$ is referred to as a  \textit{semicircular} element.

\item If $x$ is the sum of an orthogonal (semicircular) element and a constant element from $\B$, then $x$ is called a \textit{shifted orthogonal} (semicircular) element.

\item Given $m\in \N$ with $m\geq 3$, if  for each $k\in \N$, $\dd\in \{1,*\}^k$ and $b_1,\cdots,b_k\in \B$, we have 
$$\kappa^{\pi}(x^{d_1}b_1,\cdots,x^{d_k}b_k)=0$$
unless $\pi\in  NC^{\dd,m}(k),$  then we say that $x$ is a \textit{$m$-unitary} element.

\item If for each $k\in \N$, $\dd\in \{1,*\}^k$ and $b_1,\cdots,b_k\in \B$, we have 
$$\kappa^{\pi}(x^{d_1}b_1,\cdots,x^{d_k}b_k)=0$$
unless $\pi\in  NC^{\dd,\infty}(k),$  then we say that $x$ is a \textit{free-unitary} element.

\item If  for each $k\in \N$, $\dd\in \{1,*\}^k$ and $b_1,\cdots,b_k\in \B$, we have 
$$\kappa^{\pi}(x^{d_1}b_1,\cdots,x^{d_k}b_k)=0$$
unless $\pi\in  NC^{\dd}(k),$  then we say $x$ is an \textit{R-diagonal} elements 

\item  If  for each $k\in \N$, $\dd\in \{1,*\}^k$ and $b_1,\cdots,b_k\in \B$, we have 
$$\kappa^{\pi}(x^{d_1}b_1,\cdots,x^{d_k}b_k)=0$$
unless $\pi\in  NC_2^{\dd}(k),$  then we say $x$ is a \textit{circular} element.

\item If $x$ is the sum of a circular element and a constant element from $\B$, then $x$ is called a \textit{shifted circular} element.

\end{enumerate}
\end{definition}

Notice that the vanishing condition for cumulants associated with noncrossing partitions depends on the blocks contained within them. 
Therefore, the partitions in the above definition can be replaced by a single block of full size $k$.
We can also define shifted $R$-diagonal or shifted $k$-unitary elements; however, we will see that we do not have de Finetti theorems for them.
In classical probability, random variables are assumed to commute with each other; thus, there is no $R$-diagonal counterpart definition in classical probability.
Notice that circular and semicircular elements are of central limit laws if free probability, thus their counterparts in classical probability are real-Gaussian and complex Gaussian distributions.
Based on the properties of the free product of probability spaces, in the scalar case, the probability space generated by freely independent self-adjoint random variables is automatically tracial. Consequently, the $W^* $-probability space is never of Type III.

Scalar-valued circular elements were introduced by Voiculescu in \cite{Vo2} via creation operators on Fock space, and the operator-valued case was then considered in \cite{Vo3}. 
Since circular elements have the most vanishing conditions on cumulants, the free additive convolution of any types of random variables in Definition \ref{types of random variables} with a circular elements remains of the same type.
This can be used to construct many nontrivial examples of random variables, as defined in Definition \ref{types of random variables}, especially in Type III probability spaces \cite{Sh1}. 
The R-diagonal elements introduced in \cite{NS1} were developed to approach Haar unitaries and circular elements; their operator-valued case was shown to be the limit law of random Vandermonde matrices \cite{BD}.

\section{Finite  sequences of symmetric random variables}

Recall that the original de Finetti-type theorem states that infinite exchangeable sequences of random variables are identically distributed conditionally independent. The quantum analogue of the de Finetti-type theorem asserts that infinite quantum exchangeable sequences of real-valued random variables are  identically distributed conditionally free independent (or operator-valued freely independent) and has been generalized to hold for $*$-random variables \cite{Cu2}.

In this paper, the quantum groups lie either between permutation groups and unitary groups for commuting random variables or between quantum permutation groups and quantum unitary groups in a noncommutative framework. Therefore, for an infinite sequence of random variables, we may assume they are either identically distributed  conditionally classical independent or identically distributed  conditionally free independent. 

As a result, the classification of symmetric random variables reduces to the classification of a single operator-valued random variable. We will demonstrate that the distributional characterization from quantum groups naturally leads to a vanishing condition on cumulants. In the following definition, we begin with an arbitrary quantum group.

\begin{definition}  Let $\A$ be a compact matrix quantum group for $n$ with generators $(u_{i,j})$. 
For each $k\in\N$, $\dd=(d_1,\cdots,d_k)\in \{1,*\}^k$, we say that $\dd$ is an $\A$-sequence
if $$\sum\limits_{\alpha=1}^n u^{d_1}_{\alpha,j}u^{d_2}_{\alpha,j}\cdots u^{d_k}_{\alpha,j}=1$$ for all $j=1,\cdots,n$.
For each $m\in\N$, $\dd=(d_1,\cdots,d_m)\in \{1,*\}^k$ ,we define $NC^{\A,\dd}(m)\subset NC(m)$ to the set of noncrossing paritions $\pi$ whose blocks restrict to $\dd$ are all $\A$-sequences.
\end{definition}

For $\dd=(d_1,\cdots,d_k)\in \{1,*\}^k$ and $\ii=(i_1,\cdots, i_k) \in[n]^k$, we will denote $u^{d_1,j}_{i_1}\cdots u^{d_k}_{i_k,j} $ by $u^\dd_{\ii, j}$.
According to the definition, for each $k\in\N$, $\dd=(d_1,\cdots,d_k)\in \{1,*\}^k$ and $\pi\in NC^{\A,\dd}(k)$, we have 
$$ \sum_{\substack{\ii \in[n]^k,  \pi\leq \ker\ii }} u^\dd_{\ii, 1} =1_\A. $$

One should be cautious that $NC^{\A,\dd}(k)$ does not necessarily include all noncrossing partitions $\pi$ that satisfy the aforementioned equation.
 This observation highlights that certain universal conditions for compact matrix quantum groups do not affect the distribution of the corresponding variables. It also explains why the probabilistic classification of quantum groups based on de Finetti-type theorems is coarser than the classification for easy quantum groups.

\begin{proposition}\label{non-vanishing cumulant implies identity} Let $\A$ be a compact matrix quantum group for $n$ with generators $(u_{i,j})$, and let   $(\M,\E)$ be a $\B$-valued probability space. Assume that   $x_1,\cdots,x_n\in \M$ are  freely independent, identically distributed and  $\A$-invariant. Then, for each $k\in\N$, $\dd=(d_1,\cdots,d_k)\in \{1,*\}^k$,
the cumulant $$\kappa^{(k)}\left[b_0x^{d_1}_1, b_1 x^{d_2}_1 \cdots ,b_{k-1}x^{d_k}_1 b_k\right]=0$$ unless
$$\sum\limits_{\alpha=1}^n u^{d_1}_{\alpha,j}u^{d_2}_{\alpha,j}\cdots u^{d_k}_{\alpha,j}=1$$ for all $j=1,\cdots,n$.

\end{proposition}

\begin{proof}
By the definition of $\A$-invariant, we have
 $$\E\left[b_0x^{d_1}_1b_1\right]  =\sum_{\alpha=1}^n \E\left[b_0x^{d_1}_\alpha b_1\right] \otimes u^{d_1}_{\alpha, 1}.$$

Notice that $\kappa^{(1)}\left[b_0x^{d_1}_1b_1\right]=\E\left[b_0x^{d_1}_1b_1\right]$ and $x_i$'s are identically distributed, it follows that 
 $$\kappa^{(1)}\left[b_0x^{d_1}_1b_1\right]  =\sum_{\alpha=1}^n \kappa^{(1)}\left[b_0x^{d_1}_1b_1\right] \otimes u^{d_1}_{\alpha, 1}.$$
By comparing the coefficient of $\kappa^{(1)}\left[b_0x^{d_1}_1b_1\right]$, we have $\kappa^{(1)}\left[b_0x^{d_1}_1b_1\right]=0$ unless $\sum\limits_{\alpha=1}^n u^{d_1}_{\alpha,1}=1$. 
Since $x_i$'s are identically distributed, replacing $1$ by $j$, we get the statement is true for $k=1$.

Assume that the statement is true for $1,\cdots,k-1$. 

$$
\begin{aligned}
& \mathbb{E}\left[b_0 x^{d_1}_1 b_1 x^{d_2}_1 \cdots x^{d_k}_k b_k\right] \otimes 1 \\
= & \sum_{\ii \in[n]^k} \mathbb{E}\left[b_0 x_{i_1}^{d_1} b_1 x_{i_2}^{d_2} \cdots x_{i_k}^{d_k} b_k\right]\otimes u^\dd_{\ii, 1} \\
= & \sum_{\ii \in[n]^k} \sum_{\pi \in N C(k)} \kappa^{(\pi)}\left[b_0 x_{i_1}^{d_1}, b_1 x_{i_2}^{d_2}, \cdots, b_{k-1}x_{i_k}^{d_k} b_k\right] \otimes u^\dd_{\ii, 1}\\
= & \sum_{\pi \in N C(k)} \sum_{\ii \in[n]^k} \kappa^{(\pi)}\left[b_0 x_{i_1}^{d_1}, b_1 x_{i_2}^{d_2}, \cdots, b_{k-1}x_{i_k}^{d_k} b_k\right] \otimes u^\dd_{\ii, 1}
\end{aligned}
$$

Because that $x_i$'s are freely independent,  for a given partition $\pi$ ,  the cumulant $\kappa^{(\pi)}\left[b_0 x_{i_1}^{d_1}, b_1 x_{i_2}^{d_2}, \cdots, b_{k-1}x_{i_k}^{d_k} b_k\right]0$ does not vanish only if each block of $\pi$ contains the same $x_i$. This implies that $\ker\ii\geq \pi$. Thus, the last term of the equation becomes 

$$\sum_{\pi \in N C(k)} \sum_{\substack{\ii \in[n]^k\\  \pi\leq \ker\ii }}\kappa^{(\pi)}\left[b_0 x_{i_1}^{d_1}, b_1 x_{i_2}^{d_2}, \cdots ,b_{k-1}x_{i_k}^{d_k} b_k\right] \otimes u^\dd_{\ii, 1}$$

Again, since the $x_i$'s are identically distributed and all  blocks of each partition $\pi$ contains the same $x_i$, we can replace all the elements in the above term with $x_1$.  Therefore, we obtain

\begin{equation*}\label{MC-for x_1}
\sum_{\pi \in N C(k)} \sum_{\substack{\ii \in[n]^k\\  \pi\leq \ker\ii }}\kappa^{(\pi)}\left[b_0 x_{1}^{d_1}, b_1 x_{1}^{d_2}, \cdots, b_{k-1} x_{1}^{d_k} b_k\right] \otimes u^\dd_{\ii, 1}
\end{equation*}

By induction, all possible non-vanishing terms in the above equation correspond to those partitions $\pi \in NC^{\A,\dd}(k)$ that have at least two blocks and the whole block $1_k$. Thus, we obtain

\begin{equation*}\label{MC-for x_1}
\sum_{\pi \in NC^{\A,\dd}(k),\pi\neq 1_k} \sum_{\substack{\ii \in[n]^k\\  \pi\leq \ker\ii }}\kappa^{(\pi)}\left[b_0 x_{1}^{d_1}, b_1 x_{1}^{d_2}, \cdots, b_{k-1} x_{1}^{d_k} b_k\right] \otimes u^\dd_{\ii, 1} + \sum\limits_{\alpha=1}^n\kappa^{(k)}\left[b_0 x_{1}^{d_1}, b_1 x_{1}^{d_2}, \cdots, b_{k-1} x_{1}^{d_k} b_k\right] \otimes u^{d_1}_{\alpha,1}u^{d_2}_{\alpha,1}\cdots u^{d_k}_{\alpha,1}
\end{equation*}

For fixed $\pi \in NC^{\A,\dd}(k)$, notice that$ \sum\limits_{\ii \in[n]^k, \pi\leq \ker\ii } u^\dd_{\ii, 1} =1_\A $,  we get 
$$
\begin{aligned}
& \E\left[b_0 x^{d_1}_1 b_1 x^{d_2}_1 \cdots x^{d_k}_1 b_k\right] \otimes 1_\A \\
= &\sum_{\pi \in NC^{\A,\dd}(k),\pi\neq 1_k} \kappa^{(\pi)}\left[b_0 x_{1}^{d_1}, b_1 x_{1}^{d_2}, \cdots, b_{k-1} x_{1}^{d_k} b_k\right]\\
&+ \sum\limits_{\alpha=1}^n\kappa^{(k)}\left[b_0 x_{1}^{d_1}, b_1 x_{1}^{d_2}, \cdots, b_{k-1} x_{1}^{d_k} b_k\right] \otimes u^{d_1}_{\alpha,1}u^{d_2}_{\alpha,1}\cdots u^{d_k}_{\alpha,1}
\end{aligned}
$$

On the other hand, by the moment-cumulants formula we have 

$$ \E\left[b_0 x^{d_1}_1 b_1 x^{d_2}_1 \cdots x^{d_k}_k b_k\right] \otimes 1_\A 
= \sum_{\pi \in NC(k)} \kappa^{(\pi)}\left[b_0 x_{1}^{d_1}, b_1 x_{1}^{d_2}, \cdots, b_{k-1} x_{1}^{d_k} b_k\right]
$$

Again, by  induction on $k$ for vanishing cumulants, we have

$$
\begin{aligned}
& \mathbb{E}\left[b_0 x^{d_1}_1 b_1 x^{d_2}_1 \cdots x^{d_k}_k b_k\right] \otimes 1_A \\
= &\sum_{\pi \in NC^{\A,\dd}(k),\pi\neq 1_k} \kappa^{(\pi)}\left[b_0 x_{1}^{d_1}, b_1 x_{1}^{d_2}, \cdots, b_{k-1} x_{1}^{d_k} b_k\right]+ \kappa^{(k)}\left[b_0 x_{1}^{d_1}, b_1 x_{1}^{d_2}, \cdots, b_{k-1} x_{1}^{d_k} b_k\right] 
\end{aligned}
$$

Therefore, we have 

$$\kappa^{(k)}\left[b_0 x_{1}^{d_1}, b_1 x_{1}^{d_2}, \cdots, b_{k-1} x_{1}^{d_k} b_k\right] = \sum\limits_{\alpha=1}^n\kappa^{(k)}\left[b_0 x_{1}^{d_1}, b_1 x_{1}^{d_2}, \cdots, b_{k-1} x_{1}^{d_k} b_k\right] \otimes u^{d_1}_{\alpha,1}u^{d_2}_{\alpha,1}\cdots u^{d_k}_{\alpha,1},$$
which shows that $\kappa^{(k)}\left[b_0 x_{1}^{d_1}, b_1 x_{1}^{d_2}, \cdots, b_{k-1} x_{1}^{d_k} b_k\right]=0$ unless $\sum\limits_{\alpha=1}^n u^{d_1}_{\alpha,1}u^{d_2}_{\alpha,1}\cdots u^{d_k}_{\alpha,1}=1_\A.$

Replacing $1$ by $j$, then the proof is done.
\end{proof}

Apply Proposition \ref{quantumgroup-Block identity}, we get the free part of Theorem \ref{maximal}.
Recall the following de Finetti-type theorem for infinite quantum exchangeable sequences of $*$-random variables:
\begin{theorem}\label{curran's de Finetti}
Let $(\domain,\phi)$ be a $W^*$-probability space, $\phi$ is normal faithful and $\domain$ is generated by an infinite sequence of quantum exchangeable random variables $(x_i)_{i\in \N}$. 
Then, there is a unital $W^*$-subalgebra $\range$ and a $\phi$-preserving conditional expectation $\E:\domain\rightarrow \range$ such that $(x_i)_{i\in \N}$ are freely independent in $(\domain,\E)$ with identical $\range$-valued distributions.
\end{theorem}

Therefore, we obtain the free part of Theorem \ref{de Finetti}. To derive the classical part of Theorem \ref{maximal} and Theorem \ref{de Finetti}, it suffices to apply the classical de Finetti theorem and consider classical cumulants.  The distributional symmetry passes to the $\range$-valued distributions because that $\phi$ is faithful and $\E$ a $\phi$-preserving conditional expectation.


\section{Properties of universal relations of quantum unitary groups}

In Definition \ref{quantum groups via vanishing conditions}, the quantum groups are defined through equations with identities and  vanishing conditions. 
In this section, we will demonstrate that the vanishing conditions are redundant; that is, the equations of sums with the identity are sufficient to determine all the universal conditions. 
Consequently, the failure to be quantum subgroups implies the failure of the equation of sums over blocks. 
As a result, the associated distributional symmetries will impose additional vanishing conditions on cumulants. 
This is the fundamental reason that the quantum groups in Definition \ref{quantum groups via vanishing conditions} are maximal, in the sense that the corresponding de Finetti-type theorem fails if a sequence of random variables satisfies symmetries other than the maximal one.

To proceed, we introduce several notations on operator-valued matrices algebras.
As for quantum groups,  for a  $C^*$-algebra $\A$, we will assume that $u=(u_{i,j})$ is an element in $M_n(\A)$ whose  $i,j$-th entry is $u_{i,j}$.
In addition, we assume that both $u$ and $\bar{u}=(u_{i,j}^*)$ are unitaries in $M_n(\A)$.
Recall that in Proposition \ref{quantumgroup-Block identity}, the equations involve sums that contain products of elements with the same sub-index, resembling the Hadamard product in the scalar case. Therefore, we can establish the following definition in the operator-valued framework.

\begin{definition} Let $u,v\in M_n(\A)$ be $n\times n$ matrices over $\A$.  
The Hadamard product of $u=(u_{i,j}),v=(v_{i,j})$ is defined as the matrix $w=(w_{i,j})\in M_n(\A)$ with $w_{i,j}=u_{i,j}v_{i,j}$ and is denoted by $u\circ v$.  
The Hadamard product of $k$ copies of $u$ is denoted by $u^{\circ k}$.
\end{definition} 

The following proposition demonstrates that the Hadamard product with a unitary element is contractive.

\begin{proposition}\label{Contraction of Hadarmad product} Let $u=(u_{i,j})$ be a unitary element in $\M_n(\A)$. Then, $ \|u\circ v\|\leq \|v\|$ for any $v\in M_n(\A)$.
\end{proposition}
\begin{proof}
Let $w=u\circ v=(u_{i,j}v_{i,j})$.  Then,  $w^*=(v^*_{j,i}u^*_{j,i})_{i,j}$ and 

       $$w^*w=\left(\sum\limits_{\alpha=1}^n v^*_{\alpha,i}u^*_{\alpha,i}u_{\alpha,j}v_{\alpha,j} \right)_{i,j}$$
Since $u$ is a unitary,   for all $\alpha\in [n]$ and $i\neq j$, we have 
$$v^*_{\alpha,i}u^*_{\alpha,i}u_{\alpha,i}v_{\alpha,i}= v^*_{\alpha,i}v_{\alpha,i}- \sum\limits_{\beta\neq \alpha}^n v^*_{\alpha,i}u^*_{\beta,i}u_{\beta,i}v_{\alpha,i}$$ 
and 
  $$v^*_{\alpha,i}u^*_{\alpha,i}u_{\alpha,j}v_{\alpha,j}=-\sum\limits_{\beta\neq \alpha}^n v^*_{\alpha,i}u^*_{\beta,i}u_{\beta,j}v_{\alpha,j}$$ 

Therefore,
 $$w^*w= Diag_n(\sum\limits_{\alpha=1}^n v^*_{\alpha,i}v_{\alpha,i})_i-( \sum\limits_{\beta\neq \alpha}^n v^*_{\alpha,i}u^*_{\beta,j}u_{\beta,j}v_{\alpha,j})_{i,j}$$
 where $Diag_n(x_i)_i$ denotes the diagonal matrix $(\delta_{i,j} x_i)_{i,j}.$
 
 Let $S$ be the $n^2\times n$ matrix with entries $S_{(\beta,\alpha),j}$, where $\beta, \alpha, j=1,\cdots, n$, b defined as follows: 
 $$
S_{(\beta,\alpha),i}=\left\{\begin{array}{lr}
u_{\beta,i}v_{\alpha,i}& \text{if}\, \beta\neq \alpha\\
0& \text{if}\,  \beta= \alpha\\
\end{array}\right.
 $$

Then, we have 
$$w^*w=Diag_n\left(\sum\limits_{\alpha=1}^n v^*_{\alpha,i}v_{\alpha,i}\right)-S^*S\leq Diag_n\left(\sum\limits_{\alpha=1}^n v^*_{\alpha,i}v_{\alpha,i}\right)$$

Thus, 
$$\| w^*w\|\leq \| Diag_n(\sum\limits_{\alpha=1}^n v^*_{\alpha,i}v_{\alpha,i})\|$$

Notice that $Diag_n(\sum\limits_{\alpha=1}^n v^*_{\alpha,i}v_{\alpha,i})$ is the diagonal part of $ v^*v$, it follows that 
$$\|v\|^2=\| v^*v\|\geq \max\limits_{i=1,\cdots,n}\| \sum\limits_{\alpha=1}^n v^*_{\alpha,i}v_{\alpha,i})\|=\| Diag_n(\sum\limits_{\alpha=1}^n v^*_{\alpha,i}v_{\alpha,i})\|\geq \| w^*w\|=\|w\|^2.$$
This completes the proof.
\end{proof}

For $u=(u_{i,j})\in \M_n(\A)$, following the operations on quantum groups, we denote by  $\bar{u}=(\bar{u}_{i,j})$ with
$\bar{u}_{i,j}=u^*_{i,j}$. Consequently, we have $\bar{u}=(u^t)^*=u^{(t*)}$, where $u^t$ is the transpose of $u$ with $(i,j)$-the entry being $u_{j,i}$.

\begin{lemma}\label{invertible}
Assume that $u,\bar{u}\in \M_n(\A)$ are unitary.
Let $k\geq 2$ and $d_1,\cdots,d_k\in\{1,*\}$.  
Let $\bar{d_1},\cdots\bar{d_k}\in\{1,t*\}$ such that $\bar{d_j}=1$ if and only $d_j=1.$                       
If $\sum\limits_{\alpha=1}^n u^{d_1}_{\alpha,j}u^{d_2}_{\alpha,j}\cdots u^{d_k}_{\alpha,j}=1_{\A}$ for all $j=1,\cdots,n$, then 
\begin{enumerate}
\item $ u^{\bar{d}_2}\circ\cdots \circ u^{\bar{d}_{k}}=\overline{u^{\bar{d}_1}}.$  
\item $u^{d_2}_{i,j}\cdots u^{d_k}_{i,j}=\left(u^{d_1}_{i,j}\right)^*$ for all $i,j=1,\cdots,n.$
\item $\sum\limits_{\alpha=1}^n u^{d_2}_{\alpha,j}\cdots u^{d_k}_{\alpha,j}u^{d_1}_{\alpha,j}=1_{\A}$ for all $j=1,\cdots,n$.
\end{enumerate}
\end{lemma}

\begin{proof}
Notice that $ u^{\bar{d}_2}\circ\cdots \circ u^{\bar{d}_{k}}=(u^{d_2}_{i,j}\cdots u^{d_{k}}_{i,j})_{i,j}$,
thus the diagonal elements of $(u^{\bar{d}_1})^t(u^{\bar{d}_2}\circ\cdots \circ u^{\bar{d}_{k}})$ are $\sum\limits_{\alpha=1}^n u^{d_1}_{\alpha,j}u^{d_2}_{\alpha,j}\cdots u^{d_k}_{\alpha,j}=1$.
By proposition \ref{Contraction of Hadarmad product}, we have   $\|(u^{\bar{d}_1})^t(u^{\bar{d}_2}\circ\cdots \circ u^{\bar{d}_{k}})\|\leq 1$, it follows that 
$$ (u^{\bar{d}_1})^t (u^{\bar{d}_2}\circ\cdots \circ u^{\bar{d}_{k}})=1_{M_n(\A)}.$$

Since $(u^{\bar{d}_1})^t$ is either $ \bar{u}^*$ or $u^*$, both of which are unitary,  this proves (1). 
It follows that $$(u^{\bar{d}_2}\circ\cdots \circ u^{\bar{d}_{k}})=\left((u^{\bar{d}_1})^t\right)^*=\overline{u^{\bar{d}_1}}=\left((u^{d_1}_{i,j})^*\right)_{i,j},$$
and
$$ (u^{\bar{d}_2}\circ\cdots \circ u^{\bar{d}_{k}})(u^{\bar{d}_1})^t =1_{M_n(\A)}.$$
Notice that $(u^{\bar{d}_1})^*=(u^{\bar{d}_2}\circ\cdots \circ u^{\bar{d}_{k}})$ whose $(i,j)$-th entries are given by $u^{d_2}_{j,i}\cdots u^{d_{k}}_{j,i}$, thus (2) holds.\\ 
By checking the entries of $ (u^{\bar{d}_2}\circ\cdots \circ u^{\bar{d}_{k}})u^{\bar{d}_1}$, we get 
$$\sum\limits_{\alpha=1}^n u^{d_2}_{\alpha,j}\cdots u^{d_k}_{\alpha,j}u^{d_1}_{\alpha,j}=1_{\A}$$ 
for all $j=1,\cdots,n$.

\end{proof}

Recall that an element $a\in \A$ is called a  partial isometry  if both $aa^*$ and $a^*a$ are projections.

\begin{lemma}\label{ABAB Case} Assume that $u,\bar{u}\in \M_n(\A)$ are unitary. If $\sum\limits_{\alpha=1}^n u^*_{\alpha,j}u_{\alpha,j}u^*_{\alpha,j}u_{\alpha,j}=1_{\A}$ for all $j=1,\cdots,n$,
then we have that 
\begin{enumerate}
\item $\sum\limits_{\alpha=1}^n u_{\alpha,j}u^*_{\alpha,j}u_{\alpha,j}u^*_{\alpha,j}=1$ for all $j=1,\cdots,n$
\item All $u_{i,j}$'s are partial isometries 
\item  $u^*_{i,k}u_{j,k}=u_{i,k}u^*_{j,k}=0$ for all $k=1,\cdots,n$, $i\neq j$
\end{enumerate} 
\end{lemma}
\begin{proof}
\begin{enumerate}
\item 
This is the case $k=4$, $(d_1,\cdots,d_4)=(*,1,*,1)$ in Lemma \ref{invertible}, thus we have 
$$\sum\limits_{\alpha=1}^n u_{\alpha,j}u^*_{\alpha,j}u_{\alpha,j}u^*_{\alpha,j}=1,$$ for all $j=1,\cdots,n$.

\item 
Notice that $u_{\alpha,j}u^*_{\alpha,j}$'s are positive and 
$$u_{\alpha,j}u^*_{\alpha,j}\leq 1_\A.$$
Hence, we have  $u_{\alpha,j}u^*_{\alpha,j}u_{\alpha,j}u^*_{\alpha,j}\leq u_{\alpha,j}u^*_{\alpha,j}$. 
However, we  also have 
$$1_{\A}=\sum\limits_{\alpha=1}^n u^*_{\alpha,j}u_{\alpha,j}u^*_{\alpha,j}u_{\alpha,j}\leq \sum\limits_{\alpha=1}^n u^*_{\alpha,j}u_{\alpha,j}=1_{\A}.$$
It follows that $u^*_{\alpha,j}u_{\alpha,j}=u^*_{\alpha,j}u_{\alpha,j}u^*_{\alpha,j}u_{\alpha,j}$. 

Therefore, $u^*_{\alpha,j}u_{\alpha,j}$'s are projections.  Similarly,  $u_{\alpha,j}u^*_{\alpha,j}$'s  are also projections.

\item  Since $ 1_\A=\sum\limits_{\gamma=1}^n u_{\alpha,\gamma}^*, u_{\alpha, \gamma}$ for $\alpha=1,\cdots,n,$ and  $u_{\alpha,j}u^*_{\alpha,j}$'s  are also projections,  we have 
$$ 
\begin{aligned}
n 1_\A & =\sum_{\alpha=1}^n \sum_{\beta=1}^n u_{\alpha, \beta} u_{\alpha, \beta}^*  \\
&=\sum_{\alpha=1}^n \sum_{\beta=1}^n u_{\alpha, \beta} u_{\alpha, \beta}^* u_{\alpha, \beta} u_{\alpha, \beta}^* \\
& =\sum_{\alpha=1}^n \sum_{\beta=1}^n\left(u_{\alpha, \beta}\left(1_\A-\sum_{\substack{\gamma=1\\ \gamma \neq \beta}}^n u_{\alpha,\gamma}^*, u_{\alpha, \gamma}\right) u_{\alpha, \beta}^*\right) \\
& =n 1_\A-\sum_{\substack{\alpha, \beta, \gamma=1 \\ \gamma \neq \beta}}^n u_{\alpha, \beta} u_{\alpha, \gamma}^* u_{\alpha, \gamma} u_{\alpha, \beta}^* \\
&
\end{aligned}
$$

Since all $ u_{\alpha, \beta} u_{\alpha, \gamma}^* u_{\alpha, \gamma} u_{\alpha, \beta}^* $ are positive, they must all zero which means $u_{\alpha, \gamma} u_{\alpha, \beta}^*=0 $ for all $\alpha,\beta=1,\cdots,n$. 

Similarly, we also have $u^*_{\alpha, \gamma} u_{\alpha, \beta}=0 $ for all $\alpha,\beta=1,\cdots,n$.

\end{enumerate}
\end{proof}

\begin{lemma}\label{AABB Case} Assume that $u,\bar{u}\in \M_n(\A)$ are unitary. If $\sum\limits_{\alpha=1}^n u^*_{\alpha,j}u^*_{\alpha,j}u_{\alpha,j}u_{\alpha,j}=1$ for all $j=1,\cdots,n$,
then 
\begin{enumerate}
\item all $u_{i,j}$'s are partial isometries 
\item  $u^*_{i,k}u_{j,k}=u_{i,k}u^*_{j,k}=0$ for all $k=1,\cdots,n$, $i\neq j$
\item $u_{i,j}u^*_{i,j}=u^*_{i,j}u_{i,j}$ for all $i,j=1,\cdots,n$
\end{enumerate} 
\end{lemma}
\begin{proof}
This is the case $k=4$, $(d_1,\cdots,d_4)=(*,*,1,1)$ in Lemma \ref{invertible}, we have 
$u_{i,j}=u^*_{i,j}u_{i,j}u_{i,j}$ and thus $u^*_{i,j}=u^*_{i,j}u^*_{i,j}u_{i,j}$
for $i,j=1,\cdots,n.$
Thus, we have 
$$u_{i,j}u^*_{i,j}=u^*_{i,j}u_{i,j}u_{i,j}u^*_{i,j} $$
and 
$$u_{i,j}u^*_{i,j}=u_{i,j}u^*_{i,j}u^*_{i,j}u_{i,j}$$
By Lemma \ref{invertible}, we have $\sum\limits_{\alpha=1}^n u^*_{\alpha,j}u_{\alpha,j}u_{\alpha,ju^*_{\alpha,j}}=1$ for all $j=1,\cdots,n$, so we get
$$u^*_{i,j}u_{i,j}=u^*_{i,j}u_{i,j}u_{i,j}u^*_{i,j} $$
and 
$$u^*_{i,j}u_{i,j}=u_{i,j}u^*_{i,j}u^*_{i,j}u_{i,j}.$$

Therefore, we have 
$$u_{i,j}u^*_{i,j}=u^*_{i,j}u_{i,j}=u_{i,j}u^*_{i,j}u^*_{i,j}u_{i,j}=\left( u^*_{i,j}u_{i,j}\right)^2.$$
It implies statement (1) and (3).  Statement (2) follow the proof in Lemma \ref{ABAB Case}. 
\end{proof}

The two lemmas will be applied to $\rqh$ and $\qh_0$, respectively. 
The main difference between Lemma \ref{ABAB Case} and Lemma \ref{AABB Case} is that the entries satisfying the equations in Lemma \ref{AABB Case} must be normal, a condition that does not occur in Lemma \ref{ABAB Case}.

\begin{proposition} \label{sum of power m}
Assume that $u,\bar{u}\in \M_n(\A)$ are unitary.
Let $\dd \in \{1,*\}^k$, with $k\geq 3$. If $\sum\limits_{\alpha=1}^n u^{d_1}_{\alpha,j}u^{d_2}_{\alpha,j}\cdots u^{d_k}_{\alpha,j}=1_\A$ for all $j=1,\cdots,n$, then
\begin{enumerate}
\item all $u_{i,j}$'s are partial isometries 
\item  $u^*_{i,k}u_{j,k}=u_{i,k}u^*_{j,k}=0$ for all $k=1,\cdots,n$, $i\neq j$
\end{enumerate} 
Further, if $k$ is not even or $\bar{d}$ is not alternating, then $u_{i,j}u^*_{i,j}=u^*_{i,j}u_{i,j}$ for all $i,j=1,\cdots,n$.

\end{proposition}
\begin{proof}By Lemma \ref{invertible}, we know that $w=(u^{d_2}_{i,j}\cdots u^{d_k}_{i,j})_{i,j}=$ is a unitary. 
Let $v=(u^{d_{k-1}}_{i,j}u^{d_k}_{i,j})_{i,j}.$
By Proposition \ref{Contraction of Hadarmad product}, we have $\|v\|\leq 1$. 
Since $u$ and $w$ are unitaries, we have $u_{i,j}u^*_{i,j}$ and $u^*_{i,j}u_{i,j}$ are positive operators less than or equal to $1_{\A}$. 
Therefore, we have 
$$
1_\A=\sum\limits_{\alpha=1}^n \left(u^{d_2}_{\alpha,j}\cdots u^{d_k}_{\alpha,j}\right)^*\left(u^{d_2}_{\alpha,j}\cdots u^{d_k}_{\alpha,j}\right)\leq \sum\limits_{\alpha=1}^n \left(u^{d_{k-1}}_{\alpha,j}u^{d_k}_{\alpha,j}\right)^*\left(u^{d_{k-1}}_{\alpha,j} u^{d_k}_{\alpha,j}\right),
$$
the last term is an element on the diagonal of $v^*v$. Since $\|v^*v\|=\|v\|^2\leq 1$, we have  
$$\sum\limits_{\alpha=1}^n \left(u^{d_{k-1}}_{\alpha,j}u^{d_k}_{\alpha,j}\right)^*\left(u^{d_{k-1}}_{\alpha,j} u^{d_k}_{\alpha,j}\right)=1_\A,$$ which is either the case in Lemma \ref{ABAB Case} and Lemma \ref{AABB Case}.  
It follows that all $u_{i,j}$'s are partial isometries and $u^*_{i,k}u_{j,k}=u_{i,k}u^*_{j,k}=0$ for all $k=1,\cdots,n$, $i\neq j$.

If $1$ and $*$ do not appear alternately in $\dd$, then there exists $1\leq i\leq k-1$ such that $d_i=d_{i+1}$. By the cyclic statement from Lemma \ref{invertible},
we may assume that $d_{k-1}=d_k$.  From the analysis of $v^*v$ above, we find the situation described in  Lemma \ref{ABAB Case}. Therefore, we have $u_{i,j}u^*_{i,j}=u^*_{i,j}u_{i,j}$ for all $i,j=1,\cdots,n$.

If $1$ and $*$  appear alternately in $\dd$ and $k$ is odd, then $d_1=d_k$. It follows that  $\sum\limits_{\alpha=1}^n u^{d_2}_{\alpha,j}\cdots u^{d_k}_{\alpha,j}u^{d_1}_{\alpha,j}=1_\A$ for all $j=1,\cdots,n$. This is again the situation described in  Lemma \ref{ABAB Case}.  The proof is done.

\end{proof}

\begin{corollary}\label{vanishing case}
Assume that $u,\bar{u}\in \M_n(\A)$ are unitary. Given $k\in \N$ and   $\sum\limits_{\alpha=1}^n u^{d_1}_{\alpha,j}u^{d_2}_{\alpha,j}\cdots u^{d_k}_{\alpha,j}=1$ for all $j=1,\cdots,n$.
Let $m$ be the difference of the numbers of $1$ and $*$ appeared in the sequence $(d_1,\cdots,d_k).$  
If $m\geq 1$, then $$\sum\limits_{\alpha=1}^n u^m_{\alpha,j}=1.$$
for all $j=1,\cdots,n$
\end{corollary}
\begin{proof}
The statement is trivial for $k=1$. For $k=2$, since $m\geq 1$, we have $\dd=(1,1)$ or $(*,*).$  The statement is true.

For $k>2$, if $m\geq 2$,  then $\dd$ is never alternating. By Proposition \ref{sum of power m}, all $u_{i,j}$'s are normal partial isometries.  If the numbers of $1$ is greater than the number of $*$, then we have  
$$u^{d_1}_{\alpha,j}u^{d_2}_{\alpha,j}\cdots u^{d_k}_{\alpha,j}=u^m_{\alpha,j}.$$
Otherwise we have 
$$u^{d_1}_{\alpha,j}u^{d_2}_{\alpha,j}\cdots u^{d_k}_{\alpha,j}=\left(u_{\alpha,j}^*\right)^m.$$
In both cases, we have 
$$\sum\limits_{\alpha=1}^n u^m_{\alpha,j}=1.$$
If $m=1$ and $\dd$ is not alternating, then $1$ and $*$  appear alternately in $\dd$ and $k$ is odd, the statement holds true.
\end{proof}

\begin{proposition} Let $\A(n)$ be a quantum unitary group generated by $\{u_{i,j}|i,j=1,\cdots,n\}$. 

 \begin{enumerate}
  
\item If $\sum\limits_{\alpha=1}^n u^2_{\alpha, i}=1_{\A(n)} $ for all $i=1,\cdots,n$, then $\A(n)\subset C(\qo,n)$.

\item  If $\sum\limits_{\alpha=1}^n u^2_{\alpha, i}=1_{\A(n)} $ and $\sum\limits_{i=1}^n u_{i,k}=1_{\A(n)} $  for all $i=1,\cdots,n$, then $\A(n)\subset C(\qb_s,n)$.

\item If $\sum\limits_{\alpha=1}^n u^2_{\alpha, i}=1_{\A(n)} $ and there exists an even number $k\geq 4$ such that $\sum\limits_{\alpha=1}^n u^k_{\alpha,i}=1_{\A(n)} $  for all $i=1,\cdots,n$, then $\A(n)\subset C(\qh_s,n)$.

\item If $\sum\limits_{i=1}^n u_{\alpha, i}=1_{\A(n)} $  for all $i=1,\cdots,n$, then  $\A(n)\subset C(\qb,n) $. 
\item If there exist $k\in \N$, $(d_1,\cdots,d_k)\in \{1,*\}^k$ such that $\sum\limits_{\alpha=1}^n u^{d_1}_{\alpha,i}u^{d_2}_{\alpha,i}\cdots u^{d_k}_{\alpha,i}=1$ for all $j=1,\cdots,n$ and  the difference $m$ of the numbers of $1$ and $*$ appeared in the sequence $(d_1,\cdots,d_k)$ is greater than or equal to $3$, then $\A(n)\subset C(\qh_m,n)$.
\item If $\sum\limits_{\alpha=1}^n u_{\alpha, i} u_{\alpha, i} u_{\alpha, i}^* u_{\alpha, i}^*=1_{\A(n)}$ for all $i$, then $\A(n)\subset C(\qh_0,n)$
\item If $\sum\limits_{\alpha=1}^n u_{\alpha, i} u_{\alpha, i}^* u_{\alpha, i}  u_{\alpha, i}^*=1_{\A(n)}$ for all $i$, then$\A(n)\subset C(\rqh,n)$.
\end{enumerate}  
\end{proposition}

\begin{proof}

By Lemma \ref{invertible} (2), $u_{i,j}$'s are selfadjoint. Together with the unitary conditions, we have 
$\sum\limits_{k=1}^n u_{k, i} u_{k, j} =\delta_{i, j} $ which is the universal condition for $C(\qo,n)$. Therefore, (1) holds.  (2) and (3) follow the classification presented in \cite{Liu3}. (4) is straight from the definition of $C(\qb,n)$. (5) follows corollary \ref{vanishing case} and  Proposition \ref{sum of power m}.  (6) follows Lemma \ref{AABB Case} and (7) follows Lemma \ref{ABAB Case}.

\end{proof}

\section{Proof to Theorem 3}

First, we will show that the $\qu$  provides the maximal distribution for quantum random variables in Theorem \ref{all}.

\begin{lemma} Let $(\M,\phi)$ be a $W^*$-probability space,$\phi$ is normal faithful and $\M$ is generated by an infinite sequence of non-zero random variables $(x_i)_{i\in \N}$. If $\qs\subset\mathcal{DS}(X,\phi)$, then $\mathcal{DS}(X,\phi)\subset \qu$.
\end{lemma}
\begin{proof}
Since $\qs\subset\mathcal{DS}(X,\phi)$,  there exits a unital $W^*$-subalgebra $\range$ and a $\phi$-preserving conditional expectation $\E:\domain\rightarrow \range$ such that $(x_i)_{i\in \N}$ are freely independent in $(\domain,\E)$ with identical $\range$-valued distributions. 
Following the faithfulness of $\phi$ and $\phi$-preserving condition, the distributional symmetry is transferred  to the operator-valued probability space $(\M,\E)$. 
Thus, we have either  $\E[x_1]\neq 0$ or $\E[x_1]=0$ and $\E[x_1bx_2]=0$ for all $b\in \B$. By Proposition \ref{unitary condition}, $\A \subset C(\qu, n)$.
\end{proof}

Now, we are ready to prove Theorem \ref{all}. We will focus solely on the free case here, as the classical case follows straightforwardly by introducing a commuting relation within the framework.

\begin{proof}[Proof to Theorem 3]

It suffices to consider the non-vanishing conditions of $\range$-valued free cumulants of $x_i$.
Since $(x_i)_{i\in \N}$ is $\qu$-invariant, we do not need to consider $\kappa(x_i,x_i^*)$ and $\kappa(x_i^*,x_i)$.
If there exist $k\geq 1$ and $\dd\in\{1, *\}^k$ such that 
$$\kappa^k_{\E}(x_1^{d_1},\cdots,x_1^{d_k})\neq 0,$$
then by Proposition \ref{non-vanishing cumulant implies identity}, 
$$\sum\limits_{\alpha=1}^n u^{d_1}_{\alpha,j}u^{d_2}_{\alpha,1}\cdots u^{d_k}_{\alpha,1}=1$$.

This implies that, for all $n $, the maximal distributional quantum group $ \mathcal{A}_n$ must be one of the groups listed in the theorem. The only remaining point to prove is to describe the intersection of different quantum groups. 
Recall the quantum subgroup relations depicted in the following diagram:
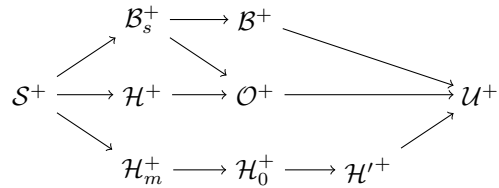
\begin{center}

\begin{tikzpicture}[xscale=1.5,yscale=1]

\path 
node   (m11) at (0,2) {$\qs$}         
node   (m21) at (1,3) {$\qb_s$}
node   (m22) at (1,2) {$\qh$}
node   (m23) at (1,1) {$\qh_m$}
node   (m33) at (2,1) {$\qh_0$}
node   (m32) at (2,2) {$\qo$}
node   (m31) at (2,3) {$\qb$}     
node   (m42) at (4,2) {$\qu$}
node   (m43) at (3,1) {$\rqh$}                               
 ;
{[->]       
\draw   (m11) --  (m21); 
\draw   (m11) --  (m22);
\draw   (m11) --  (m23); 
\draw   (m21) --  (m31);
\draw   (m21) --  (m32);
\draw   (m23) --  (m33);
\draw   (m32) --  (m42);
\draw   (m33) --  (m43);
\draw   (m31) --  (m42);
\draw   (m43) --  (m42);
\draw   (m22) --  (m32);
      }
   
\end{tikzpicture}

\end{center}

If $\mathcal{A}_n$ is a quantum subgroups of $C(\qo,n)$, from the work in \cite{Liu3}, $\mathcal{A}_n$ must  be one of  $C(\qo,n), C(\qh,n), C(\qb_s,n)$, and $C(\qo,n)$.

If $\mathcal{A}_n$ is a quantum subgroup of $C(\qb,n)$ and $C(\rqh,n)$,
then we have $\sum\limits_{k=1}^n u_{k1}=1_\A$, $\sum\limits_{k=1}^n u_{k,1}u_{k,1}^*=1_\A$ and  
the elements $u_{k,1}$'s are partial isometries with orthogonal initial spaces. 
Let $p$ be the projection onto the initial space of $u_{k,1}$.  Then, we have 
$$u_{k,1}=u_{k,1}p=\sum\limits_{k=1}^n u_{k1}p=1_\A p=p.$$
Similarly,  all $u_{i,j}$'s are projections.  
Consequently, $\mathcal{A}_n$ is a quantum subgroup of $C(\qs,n)$, which means$ \mathcal{A}_n=C(\qs,n)$.

Finally, if $\mathcal{A}_n$ is a quantum subgroup of $C(\qh_{m_1},n)$ and $C(\rqh_{m_2},n)$, with $m_1>m_2\geq 3$,
then we have 
$$\sum\limits_{k=1}^n u^{m_1}_{k1}=\sum\limits_{k=1}^n u^{m_2}_{k1}=1_\A.$$
Notice that in this case $u^{m_2}_{k1}$'s are projections and $u_{k1}$'s are normal, we have 
$$\sum\limits_{k=1}^n u^{m_1-m_2}_{k1}=1_\A.$$
Thus, $\A$ must be a quantum subgroup of $C(\qh{m}, n) $ with $m = \gcd(m_1, m_2) $ being the greatest common divisor of $m_1$and  $m_2$.\\
If $m=1$, this reduces to the case for subgroups of  $C(\qb,n)$ and $C(\rqh,n)$,  given us $\mathcal{A}_n=C(\qs,n)$.\\
If $m=2$, it leads to the selfadjoint case, which is quantum subgroup of $C(\qo,n)$.  \\
The proof is complete.
\end{proof}

\vspace{1cm}

\noindent{\bf Acknowledgment:} This work was supported by Zhejiang Provincial Natural Science Foundation of China under
Grant No. LR24A010002 and  National Natural Science Foundation of China under Grant No. 12171425.

\bibliographystyle{plain}

\bibliography{references}

\vspace{1cm}

\noindent School  of Mathematics\\
Zhejiang University	\\
Hangzhou, Zhejiang 310058, China\\
E-MAIL: lwh.math@zju.edu.cn \\

\end{document}